\documentclass[a4paper,10pt]{amsart}

\usepackage[utf8]{inputenc}
\usepackage[english,ngerman]{babel}
\usepackage{array,tikz,listings,color,graphicx,csquotes,libertine,datetime,enumitem}
\usepackage[hypertexnames=false]{hyperref}
\usetikzlibrary{shapes.misc} 
\usetikzlibrary{arrows.meta} \tikzset{>={Latex[width=2mm,length=2mm]}} 
\usetikzlibrary{decorations.pathreplacing}

\numberwithin{equation}{section}
\numberwithin{table}{section}
\makeatletter
\@addtoreset{section}{part} 
\@addtoreset{equation}{part}
\@addtoreset{figure}{part}
\@addtoreset{table}{part}
\@addtoreset{footnote}{section}
\makeatother
\newcommand\invisiblepart[1]{%
  \refstepcounter{part}%
  \addcontentsline{toc}{part}{#1}%
}
\makeatletter
\let\@cleartopmattertags\relax
\makeatother

\newif\ifen
\newcommand{\en}[1]{\ifen#1\fi} 
\newcommand{\de}[1]{\ifen\else#1\fi} 

\newtheorem{defi}{\en{Definition}\de{Definition}}[section]
\newtheorem{lem}[defi]{\en{Lemma}\de{Lemma}}

\newtheorem{thm}[defi]{\en{Proposition}\de{Satz}}

\usepackage{mdframed}
 	\definecolor{lightergray}{rgb}{0.9, 0.9, 0.9}
\newmdtheoremenv[innertopmargin=0pt, backgroundcolor=lightergray, linecolor=gray, linewidth=1pt, topline=false, bottomline=false]{theo}[defi]{\en{Theorem}\de{Theorem}}

\usepackage{fancyhdr}
\makeatletter
  \def\section{\@startsection{section}{1}%
    \z@{.7\linespacing\@plus\linespacing}{.6\linespacing}%
    {\Large\normalfont\scshape\bfseries\centering}}
\makeatother

\makeatletter
  \def\subsection{\@startsection{subsection}{1}%
    \z@{.7\linespacing\@plus\linespacing}{.6\linespacing}%
    {\large\normalfont\scshape\bfseries\centering}}
\makeatother

\newcommand{\kommentInh}[2]{{\noindent\textbf{\ref{\en{EN}#1}.~~~\nameref{\en{EN}#1}~~\dotfill~~\pageref{\en{EN}#1}}\par\vspace*{1.5pt}
                            \noindent\narrower\narrower{\itshape #2 }\par\vspace*{7pt}}}

\mdfdefinestyle{myquote}
  {
    hidealllines = true ,
    leftline, bottomline, rightline
  }
\newmdenv[style=myquote]{myquote}

\newcommand{\FM}{\operatorname{\en{LM}\de{FM}}} 

\newcommand*\circled[1]{\tikz[baseline=(char.base)]{
    \node[shape=circle,draw,inner sep=1.2pt] (char) {#1};}}
    
\definecolor{darkergray}{rgb}{0.6, 0.6, 0.6}

\newcommand{\dotline}{\hbox to \textwidth{\leaders\hbox to 3pt{\hss\textcolor{darkergray}{$\cdot$}\hss}\hfil}}

\begin{document}

    \title{{\Large{The Transcendence of $\pi$\\and the Squaring of the Circle}}\\[2ex]\hrule~\\[2ex]
        \Large{Die Transzendenz von $\pi$\\und die Quadratur des Kreises}}
    \author{Lorenz Milla, 23.05.2020}

  \entrue
  \invisiblepart{English version}
  
\selectlanguage{ngerman}
\vspace*{0cm}
\maketitle
\thispagestyle{empty}
\renewcommand{\proofname}{\en{Proof}\de{Beweis}}
\label{\en{EN}titlepage}

\vspace*{1.5cm}

{\narrower\narrower\narrower{
\noindent{\scshape Abstract.} In this paper we prove the transcendence of $\pi$ using Hilbert's method. We also prove that all points constructible with compass and straightedge have algebraic coordinates.
Thus we give a self-contained proof that squaring the circle is impossible, requiring only basic linear algebra, analysis and Cauchy's Integral Theorem.

\vspace*{2mm}
{\begin{center}\emph{English version: pp.~\pageref{ENtitlepage}--\pageref{ENende}}\end{center}}

\vspace*{1cm}

\noindent{\scshape Zusammenfassung.} In diesem Aufsatz beweisen wir mit Hilberts Methode, dass $\pi$ transzendent ist. Weiter beweisen wir, dass alle mit Zirkel und Lineal konstruierbaren Punkte algebraische Koordinaten haben.
Somit beweisen wir, dass die Quadratur des Kreises unmöglich ist. Der vorliegende Beweis ist in sich abgeschlossen und setzt nur grundlegende lineare Algebra und Analysis und den Cauchy'schen Integralsatz voraus.

\vspace*{2mm}
{\begin{center}\emph{Deutsche Version: S.~\pageref{titlepage}--\pageref{ende}}\end{center}}

~
}}
\vspace*{1cm}

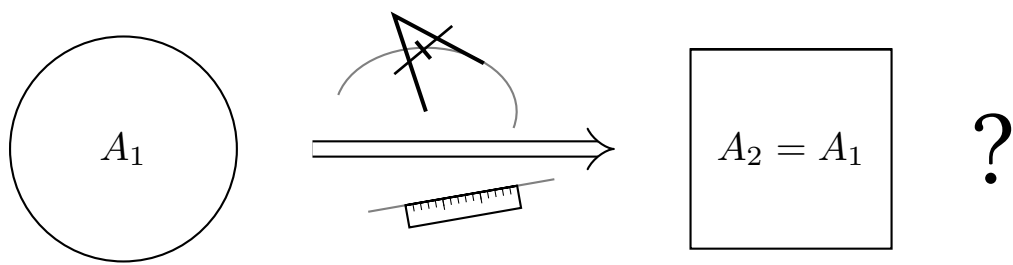
\begin{center}\scalebox{0.88}{
\begin{tikzpicture}
\draw[white] (-2.5,2.2) -- (-2.5,-4.1); 
\draw[thick] (-0.5,0) circle (1.5cm); 
\draw[-Implies,thick,double distance=5pt]  (2,0) -- (6,0); 
\draw[thick] (7,1.32934) -- (9.65868,1.32934) -- (9.65868,-1.32934) -- (7,-1.32934) -- (7,1.32934) -- (9.65868,1.32934); 
\node[scale=1.5] at (-0.5,0) {$A_1$}; 
\node[scale=1.5] at (8.32934,0) {$A_2=A_1$}; 
\node[scale=4] at (11,0) {?}; 
\begin{scope}[rotate=10,xshift=3.8cm, yshift=-1.3cm] 
  \draw[thick, gray] (-1.25,0) -- (1.25,0);
  \draw[line width=0.8pt] (-0.75,0) -- (0.75,0) -- (0.75,-0.3) -- (-0.75,-0.3) -- (-0.75,0) -- (0.75,0) ; 
  \foreach \x in {-7,...,6}
      \draw (\x/10+.05,0)--(\x/10+.05,-0.1); 
  \draw (-.25,0)--(-.25,-0.15); 
  \draw (.25,0)--(.25,-0.15); 
 
\end{scope}
\tikzset{ partial ellipse/.style args={#1:#2:#3}{ insert path={+ (#1:#3) arc (#1:#2:#3) } } }
\draw[thick, gray] (3.5,0.5) [partial ellipse=-15:165:1.2cm and 0.85cm];
\begin{scope}[xshift=4cm,yshift=1.75cm,rotate around={40:(-0.5,-1.25)}]
  \draw[line width=1.5pt] (-0.5,-1.25) -- (0,0) -- (0.5,-1.25); 
  \draw[line width=1pt] (-0.5,-0.6) -- (0.5,-0.6); 
  \draw[line width=1.5pt] (0,-0.45) -- (0,-0.75); 
\end{scope}
\node[scale=1.25] at (4,-3) {$\displaystyle \forall n\in\mathbb N~\forall(a_0,\ldots,a_n)\in\mathbb Z^{n+1}:~~a_n\neq 0~\Rightarrow~\sum_{k=0}^n a_k\cdot\pi^k \neq 0$};
\node[scale=4] at (11,-3) {!}; 
\end{tikzpicture}}\end{center}

\en{\selectlanguage{english}}\de{\selectlanguage{ngerman}}

\pagebreak

\section*{\en{Introduction and Historical Overview}\de{Einleitung und Historischer Überblick}}\label{\en{EN}kapEinleitung}
\renewcommand{\leftmark}{\en{Introduction and Historical Overview}\de{Einleitung und Historischer Überblick}}
\fancyhead[OL]{\emph{\leftmark}}
\en{Is it possible to construct a square with the same area as a given circle, using only finitely many steps with compass and straightedge?}%
\de{Kann man in endlich vielen Schritten mit Zirkel und Lineal ein Quadrat konstruieren, das den gleichen Flächeninhalt hat wie ein gegebener Kreis?}

\en{Around 430 BC, the Greek philosopher Anaxagoras was busied with this problem called ``Squaring the Circle'' while in prison (as Plutarch reports in \cite[p.~570-571~(607F)]{\en{EN}Plutarch}).}

\de{Dieses Problem, auch \glqq Quadratur des Kreises\grqq~genannt, beschäftigte schon um 430 v. Chr. den griechischen Philosophen Anaxagoras bei einem Gefäng\-nisaufenthalt (wie Plutarch in \cite[S.~570-571~(607F)]{\en{EN}Plutarch} berichtet).}

\en{More than 2300 years later, Ferdinand von Lindemann proved in 1882 that this is impossible. His proof in \cite{\en{EN}Lindemann} commences as follows (translated; German version see p.~\pageref{kapEinleitung}):}%
\de{Nach über 2300 Jahren bewies Ferdinand von Lindemann im Jahr 1882, dass dies \linebreak un\-mög\-lich ist. Sein Beweis in \cite{\en{EN}Lindemann} beginnt wie folgt:}

\medskip
\begin{quote}
    \en{``Given the futility of the extraordinarily numerous attempts to square the circle with a compass and a straightedge, it is generally considered impossible to solve the problem;}%
    \de{\glqq Bei der Vergeblichkeit der so ausserordentlich zahlreichen Versuche, die Quadratur des Kreises mit Cirkel und Lineal auszuführen, hält man allgemein die Lösung der bezeichneten Aufgabe für unmöglich;}
    \en{so far, however, there has been no evidence of this impossibility; only the irrationality of $\pi$ and $\pi^2$ is established.\par}%
    \de{es fehlte aber bisher ein Beweis dieser Unmöglichkeit; nur die Irrationalität von $\pi$ und von $\pi^2$ ist festgestellt.}
    \en{Every construction that can be carried out with compass and straightedge can be traced back to the solution of linear and quadratic equations with the help of algebraic clothing, i.e.~also to the solution of a series of quadratic equations whose first has rational coefficients, while the coefficients of each subsequent one only have such irrational numbers included, which are introduced by solving the previous equations.}%
    \de{Jede mit Cirkel und Lineal ausführbare Construction lässt sich mittelst algebraischer Einkleidung zurückführen auf die Lösung von linearen und quadratischen Gleichungen, also auch auf die Lösung einer Reihe von quadratischen Gleichungen, deren erste rationale Zahlen zu Coefficienten hat, während die Coefficienten jeder folgenden nur solche irrationale Zahlen enthalten, die durch Auflösung der vorhergehenden Gleichungen eingeführt sind.}
    \en{The final equation can thus be transformed by repeated squaring into an equation of even degree with rational coefficients.\par}%
    \de{Die Schlussgleichung wird also durch wiederholtes Quadriren übergeführt werden können in eine Gleichung geraden Grades, deren Coefficienten rationale Zahlen sind.}
    \en{One then expounds the impossibility of squaring the circle if one proves that \emph{the number $\pi$ cannot be the root of any algebraic equation of any degree with rational coefficients.}}%
    \de{Man wird sonach die Unmöglichkeit der Quadratur des Kreises darthun, wenn man nachweist, dass \emph{die Zahl $\pi$ überhaupt nicht Wurzel einer algebraischen Gleichung irgend welchen Grades mit rationalen Coefficienten sein kann.}}
    \en{The following attempt was made to provide this proof.''\hfill{~--~\emph{Ferdinand~von~Lindemann,~1882}}}%
    \de{Den dafür nöthigen Beweis zu erbringen, ist im Folgenden versucht worden.\grqq\hfill{~--~\emph{Ferdinand~von~Lindemann,~1882}}}
\end{quote}
\bigskip

\noindent \en{Lindemann presumably refers to the following articles:}%
\de{Lindemann bezieht sich hierbei vermutlich auf die folgenden Aufsätze:}
\smallskip

\begin{itemize}[itemsep=0.8ex]
    \item[1768:] \en{Johann Heinrich Lambert \cite{\en{EN}Lambert} proves the irrationality of $\pi$ and conjectures the transcendence of $\pi$.}%
    \de{Johann Heinrich Lambert \cite{\en{EN}Lambert} beweist die Irrationalität von $\pi$ und vermutet die Transzendenz von $\pi$.}
    \item[1837:] \en{Pierre Laurent Wantzel \cite{\en{EN}Wantzel} analyzes which points can be constructed with compass and straightedge and proves:
    ``The desired length can be obtained by solving a series of quadratic equations whose coefficients are rational functions of the given quantities and of the roots of the previous equations.''}%
    \de{Pierre Laurent Wantzel \cite{\en{EN}Wantzel} untersucht, welche Punkte mit Zirkel und Lineal konstruierbar sind und beweist:
    \glqq Man erhält die gesuchte Länge durch Lösen einer Reihe quadratischer Gleichungen, deren Koeffizienten rationale Funktionen der gegebenen Größen und der Wurzeln der vorherigen Gleichungen sind.\grqq}%
    \renewcommand{\thefootnote}{\Large\fnsymbol{footnote}}
    \footnote{~\en{Translated from French:}\de{Übersetzt aus dem Französischen:}\quad\guillemotleft~\emph{L'inconnue principale du problème s'obtiendra par la résolution d'une série d'équations du second degré dont les coefficients seront fonctions rationnelles des données de la question et de racines des équations précédentes.}~\guillemotright}
    \en{\\From this, Wantzel concludes (also in \cite{\en{EN}Wantzel}) that doubling the cube and trisecting most angles is impossible.}%
    \de{\\Hieraus schlussfolgert Wantzel (ebenfalls in \cite{\en{EN}Wantzel}) die Unmöglichkeit der Würfel\-verdop\-pelung und der Winkeldreiteilung.}
    \item[1873:] \en{Charles Hermite \cite{\en{EN}Hermite1,\en{EN}Hermite2} proves the irrationality of $\pi$ and of $\pi^2$.}%
    \de{Charles Hermite \cite{\en{EN}Hermite1,\en{EN}Hermite2} beweist die Irrationalität von $\pi$ und von $\pi^2$.}
\end{itemize}

\bigskip

\noindent \en{Later, the proofs of transcendence and irrationality of $ \pi $ were further simplified:}%
\de{Später wurden die Beweise der Transzendenz und Irrationalität von $\pi$ weiter vereinfacht:}\smallskip
\begin{itemize}[itemsep=0.8ex]
    \item[1893:] \en{David Hilbert \cite{\en{EN}Hilbert} simplifies Lindemann's proof of transcendence.}%
    \de{David Hilbert \cite{\en{EN}Hilbert} vereinfacht Lindemanns Transzendenzbeweis.}
    \item[1947:] \en{Ivan Niven \cite{\en{EN}Niven} simplifies Hermite's proof of irrationality -- Niven's proof fits on one page and requires no special prerequisite knowledge.}%
    \de{Ivan Niven \cite{\en{EN}Niven} vereinfacht Hermites Irrationalitätsbeweis -- Nivens Beweis passt auf eine Seite und erfordert kein spezielles Vorwissen.}
\end{itemize}
\bigskip

\en{This paper elaborates Hilbert's proof and gives a self-contained proof that Squaring the Circle is impossible -- requiring only basic linear algebra, analysis and Cauchy's Integral Theorem.}%
\de{Der vorliegende Aufsatz arbeitet Hilberts Beweis aus und bietet einen in sich geschlossenen Beweis, dass die Quadratur des Kreises unmöglich ist -- wobei nur grundlegende Kenntnisse in linearer Algebra und Analysis sowie der Cauchy'sche Integralsatz vorausgesetzt werden.}

\vfill\pagebreak

\fancyhead[OL]{\emph{\thesection.~\leftmark}}

\section*{\en{Table of Contents}\de{Inhaltsverzeichnis}}\label{\en{EN}kapInhalt}
\renewcommand{\leftmark}{\en{Table of Contents}\de{Inhaltsverzeichnis}}

\vspace*{0.2cm}

{\noindent\textbf{\nameref{\en{EN}kapEinleitung}~~\dotfill~~\pageref{\en{EN}kapEinleitung}}\par\vspace*{8pt}}

\kommentInh{anhsymm}{\en{We formulate and prove some properties of symmetric integer polynomials. This is needed for Chapter \ref{\en{EN}kaptransz}.}%
\de{Wir formulieren und beweisen einige Eigenschaften symmetrischer ganzzahliger Polynome. Das wird für Kap.~\ref{\en{EN}kaptransz} benötigt.}}

\kommentInh{kaptransz}{\en{We prove the transcendence of $\pi$ following Hilbert's proof of 1893.}%
\de{Wir beweisen die Transzendenz von $\pi$ und folgen dabei Hilberts Beweis von 1893.}}

\kommentInh{anhalg}{\en{We prove that the set of algebraic numbers is closed under addition, subtraction, multiplication, division and the calculation of square roots.}%
\de{Wir beweisen, dass die Menge der algebraischen Zahlen abgeschlossen unter Addition, Subtraktion, Multiplikation, Division und Quadratwurzelziehen ist.}}

\kommentInh{kapkonstr}{\en{We prove that all points which can be constructed in finitely many steps with compass and straightedge have algebraic coordinates.}%
\de{Wir beweisen, dass alle Punkte, die in endlich vielen Schritten mit Zirkel und Lineal konstruierbar sind, algebraische Koordinaten haben.}}

\kommentInh{kapquadr}{\en{We prove that the Squaring of the Circle is impossible.}%
\de{Wir beweisen, dass die Quadratur des Kreises unmöglich ist.}}

{\noindent\textbf{\hyperref[\en{EN}literat]{\en{References}\de{Literatur}}~~\dotfill~~\pageref{\en{EN}literat}}}

\vspace*{5mm}
\section{\en{Symmetric Polynomials}\de{Symmetrische Polynome}}\label{\en{EN}anhsymm}
\renewcommand{\leftmark}{\en{Symmetric Polynomials}\de{Symmetrische Polynome}}
\en{In this chapter, we formulate and prove some properties of symmetric integer polynomials, which are needed for the transcendence in Ch.~\ref{\en{EN}kaptransz}.
The proof elaborates \cite{\en{EN}proofwiki}.}%
\de{In diesem Kapitel formulieren und beweisen wir einige Eigenschaften symmetrischer ganzzahliger Polynome, die für den Transzendenzbeweis in Kap.~\ref{\en{EN}kaptransz} benötigt werden.
Der Beweis arbeitet \cite{\en{EN}proofwiki} aus.}

\begin{defi}
\en{A polynomial $p\in \mathbb Z \left[{X_1, \ldots, X_n}\right]$ is called ``symmetric'', iff the following holds for all permutations $\pi$ of the numbers $\{1,2,\ldots,n\}$:}%
\de{Ein Polynom $p\in \mathbb Z \left[{X_1, \ldots, X_n}\right]$ heißt \glqq symmetrisch\grqq~genau dann, wenn für jede Permutation $\pi$ der Zahlen $\{1,2,\ldots,n\}$ gilt:}
$$ p(X_1, \ldots, X_n) = p(X_{\pi(1)}, \ldots, X_{\pi(n)})$$
\end{defi}

\begin{defi}\label{\en{EN}defielem}
\en{The ``elementary symmetric polynomials'' $\sigma_{n;k}$ are:}%
\de{Die \glqq elementarsymmetrischen Polynome\grqq~$\sigma_{n;k}$ sind:}
$$\sigma_{n;k} \left({x_1, \ldots, x_n}\right) := \sum_{1 \le i_1 < \cdots < i_k \le n} x_{i_1} \cdots x_{i_k} \qquad\text{\en{with}\de{mit}}\qquad k = 1, \ldots, n$$
\end{defi}

\begin{lem}\label{\en{EN}lemvieta}
\en{The function $f(x)=(x-z_1)\cdot(x-z_2)\cdot\ldots\cdot(x-z_n)$ has the following representation as a sum using the elementary symmetric $\sigma_{n;k}$ and its zeros $z_j\in\mathbb C$:}%
\de{Die Funktion $f(x)=(x-z_1)\cdot(x-z_2)\cdot\ldots\cdot(x-z_n)$ hat die folgende Summendarstellung mit den elementarsymmetrischen $\sigma_{n;k}$ und den Nullstellen $z_j\in\mathbb C$:}
\begin{align*}
    f(x)
    &=x^n + \sum_{k=1}^{n}(-1)^k\cdot\sigma_{n;k}(z_1,\ldots,z_n)\cdot x^{n-k}
\end{align*}
\end{lem}
\begin{proof}
\en{When multiplying out $f(x)$, every parenthesis yields either $x$ or $-z_j$.
When combining all terms containing $x^{n-k}$ (and bracketing $x^{n-k}$), we obtain all products of $k$ different factors of the form $-z_j$.
Apart from the factor $(-1)^k$, this is exactly the definition of the elementary symmetric $\sigma_{n;k}$.}%
\de{Der Beweis folgt durch Ausmultiplizieren:
Man muss aus jeder Klammer entweder $x$ oder $-z_j$ multiplizieren.
Wenn man dann alle Terme, die $x^{n-k}$ enthalten, zusammenfasst (und $x^{n-k}$ ausklammert), erhält man alle Produkte von $k$ verschiedenen Faktoren der Form $-z_j$.
Abgesehen vom Faktor $(-1)^k$ ist das genau die Definition der elementarsymmetrischen $\sigma_{n;k}$.}
\end{proof}

\begin{lem}\label{\en{EN}lemsigmaganz}
\en{Let $z_1,\ldots,z_n\in\mathbb C$ be all zeros of an integer polynomial $f\in\mathbb Z[X]$ of degree $n$ with leading coefficient $\alpha_n$.
Then it holds:}%
\de{Seien $z_1,\ldots,z_n\in\mathbb C$ alle Nullstellen eines ganzzahligen Polynoms $f\in\mathbb Z[X]$ vom Grad $n$ mit Leitkoeffizient $\alpha_n$.
Dann gilt:}
$$\alpha_n\cdot\sigma_{n;k}(z_1,\ldots,z_n)\in\mathbb Z\qquad\text{\en{for}\de{für} }k=1,\ldots,n$$
\en{In other words: every elementary symmetric expression in the zeros of an integer polynomial is integer after multiplication with its leading coefficient.}%
\de{Mit anderen Worten: Jeder elementarsymmetrische Ausdruck in den Nullstellen eines ganzzahligen Polynoms ist nach Multiplikation mit dessen Leitkoeffizient ganzzahlig.}
\end{lem}
\begin{proof}
\en{According to our premises, $f(x)$ can be factorized. Then we use Lem.~\ref{\en{EN}lemvieta}:}%
\de{Nach Voraussetzung können wir $f(x)$ in Linearfaktoren zerlegen und dann Lemma \ref{\en{EN}lemvieta} anwenden:}
\begin{align*}
    f(x)=\sum_{k=0}^{n}\alpha_k\cdot x^k
    &=\alpha_n\cdot(x-z_1)\cdot(x-z_2)\cdot\ldots\cdot(x-z_n)\\
    &=\alpha_n\cdot x^n + \alpha_n\cdot\sum_{k=1}^{n}(-1)^k\cdot\sigma_{n;k}(z_1,\ldots,z_n)\cdot x^{n-k}
\end{align*}
\en{Since $f\in\mathbb Z[X]$ is an \emph{integer} polynomial, its coefficients $\alpha_n\cdot(-1)^k\cdot\sigma_{n;k}(z_1,\ldots,z_n)$ have to be integers. This yields the statement of the lemma.}%
\de{Weil $f\in\mathbb Z[X]$ ein \emph{ganzzahliges} Polynom ist, müssen die Koeffizienten $\alpha_n\cdot(-1)^k\cdot\sigma_{n;k}(z_1,\ldots,z_n)$ ganzzahlig sein. Hieraus folgt die Aussage des Lemmas.}
\end{proof}

\begin{lem}\label{\en{EN}lemelemen}
\en{Every symmetric integer polynomial $p\in\mathbb Z[X_1,\ldots,X_n]$ can be written as an integer polynomial in the elementary symmetric polynomials $\sigma_{n;k}$.}%
\de{Jedes symmetrische ganzzahlige Polynom $p\in\mathbb Z[X_1,\ldots,X_n]$ kann als ganzzahliges Polynom in den elementarsymmetrischen $\sigma_{n;k}$ geschrieben werden.}
\end{lem}
\begin{proof}
\en{Every polynomial $p$ in $n$ variables can be written as a finite sum of monomials $c\cdot x_1^{a_1}\cdot\ldots\cdot x_n^{a_n}$ with $c\neq 0$.
We sort these monomials lexicographic by exponent, comparing at first only the exponents of $x_1$, then those of $x_2$ etc.
With the ``leading monomial'' $\FM(p)$ we denote the lexicographically largest monomial of $p$.}%
\de{Jedes Polynom $p$ in $n$ Variablen kann als Summe endlich vieler Monome der Form $c\cdot x_1^{a_1}\cdot\ldots\cdot x_n^{a_n}$ mit $c\neq 0$ geschrieben werden.
Diese Monome sortieren wir lexikografisch nach Exponenten, wobei wir zunächst nur die Exponenten von $x_1$, dann die von $x_2$ usw. vergleichen.
Mit dem \glqq führenden Monom\grqq~$\FM(p)$ bezeichnen wir das lexikografisch größte Monom von $p$.}

\begin{itemize}
    \item \en{If the polynomial $p\in\mathbb Z[X_1,\ldots,X_n]$ is \emph{symmetric},
    its leading monomial is also of the form
    $\FM(p)=c\cdot x_1^{a_1}\cdot\ldots\cdot x_n^{a_n}$,
    but then it holds}%
    \de{Wenn das Polynom $p\in\mathbb Z[X_1,\ldots,X_n]$ \emph{symmetrisch} ist,
    hat das führende Monom ebenfalls die Form
    $\FM(p)=c\cdot x_1^{a_1}\cdot\ldots\cdot x_n^{a_n}$,
    aber dann gilt zusätzlich}
    $$a_1\geq a_2\geq\ldots\geq a_n.$$
\item \en{Thus the leading monomial of the elementary symmetric $\sigma_{n;k}$ is of the form}%
    \de{Bei den elementarsymmetrischen $\sigma_{n;k}$ hat das führende Monom also die Form}
    $$\FM(\sigma_{n;k}) = x_1\cdot\ldots\cdot x_k\qquad\text{\en{with}\de{mit}}\qquad k = 1, \ldots, n.$$
\end{itemize}

\en{In order to write the symmetric polynomial $p$ using the elemementary symmetric $\sigma_{n;k}$, we start with the leading monomial of $p$:}%
\de{Um das symmetrische Polynom $p$ durch die elementarsymmetrischen $\sigma_{n;k}$ darzustellen, beginnen wir mit dem führenden Monom von $p$:}
$$\FM(p)=c\cdot x_1^{a_1}\cdot\ldots\cdot x_n^{a_n}\quad\text{\en{with}\de{mit}}\quad a_1\geq a_2\geq\ldots\geq a_n.$$
\en{Then we define the following product $\sigma$ of elementary symmetric polynomials:}%
\de{Dann bilden wir das folgende Produkt $\sigma$ von elementarsymmetrischen Polynomen:}
\begin{align}\label{\en{EN}siggma}
\sigma:=c \cdot \sigma_{n;1}^{a_1 - a_2} \cdot \sigma_{n;2}^{a_2 - a_3} \cdot \ldots \cdot \sigma_{n;n - 1}^{a_{n - 1} - a_n} \cdot \sigma_{n;n}^{a_n}
\end{align}
\en{This has the leading monomial $\FM(\sigma)$:}%
\de{Dieses hat als führendes Monom:}
\begin{align*}
    \FM(\sigma) &= c\cdot x_1^{a_1 - a_2} \cdot (x_1 x_2)^{a_2 - a_3} \cdot \ldots \cdot (x_1 \cdots x_{n - 1})^{a_{n - 1} - a_n} \cdot (x_1 \cdots x_n)^{a_n}\\
    &=c\cdot x_1^{a_1}\cdot x_2^{a_2}\cdot \ldots \cdot x_n^{a_n}
\end{align*}
\en{Thus the leading monomials of $p$ and $\sigma$ are the same and the polynomial $\bar p = p - \sigma$ has a lexicographically smaller leading monomial. By induction, and since $\bar p$ is again integer and symmetric, we can express every given symmetric polynomial $p$ as the sum of finitely many terms like in (\ref{\en{EN}siggma}), which proves the Lemma.}%
\de{also stimmen die führenden Monome von $p$ und $\sigma$ überein und das Polynom $\bar p = p - \sigma$ hat ein lexikografisch kleineres führendes Monom. Weil $\bar p$ genau wie $p$ ganzzahlig und symmetrisch ist, kann man auf diese Art nach und nach (per vollständiger Induktion) jedes gegebene symmetrische Polynom $p$ als Summe endlich vieler Terme der Form (\ref{\en{EN}siggma}) schreiben, und das Lemma ist bewiesen.}
\end{proof}

\begin{theo}\label{\en{EN}theosymm}
\en{Let $z_1,\ldots,z_n\in\mathbb C$ be all zeros of an integer polynomial $f\in\mathbb Z[X]$ of degree $n$ with leading coefficient $\alpha_n$.
Let $p\in\mathbb Z[X_1,\ldots,X_n]$ be a symmetric integer polynomial in $n$ variables, with degree $\operatorname{deg}(p)$.
Then it holds}%
\de{Seien $z_1,\ldots,z_n\in\mathbb C$ alle Nullstellen eines ganzzahligen Polynoms $f\in\mathbb Z[X]$ vom Grad $n$ mit Leitkoeffizient $\alpha_n$.
Weiter sei $p\in\mathbb Z[X_1,\ldots,X_n]$ ein symmetrisches ganzzahliges Polynom in $n$ Variablen mit Grad $\operatorname{deg}(p)$.
Dann gilt}
$$\alpha_n^{\operatorname{deg}(p)}\cdot p(z_1,\ldots,z_n)\in\mathbb Z.$$
\end{theo}
\begin{proof}
\en{Given a symmetric polynomial $p\in\mathbb Z[X_1,\ldots,X_n]$, Lemma \ref{\en{EN}lemelemen} tells that we can find a polynomial $\widehat p\in\mathbb Z[Y_1,\ldots,Y_n]$ for which it holds $p=\widehat p (\sigma_{n;1},\sigma_{n;2},\ldots,\sigma_{n;n})$.
Then Lemma \ref{\en{EN}lemsigmaganz} tells that $\alpha_n\cdot\sigma_{n;k}(z_1,\ldots,z_n)\in\mathbb Z$.
Thus it holds $\alpha_n^{\operatorname{deg}(\widehat p)}\cdot p(z_1,\ldots,z_n)\in\mathbb Z$ as well.

The definition of the $\sigma_{n;k}$ yields $\operatorname{deg}(\sigma_{n;k})\geq 1$.
Thus we have $\operatorname{deg}(\widehat p)\leq\operatorname{deg}(p)$ and the Theorem is proven.}%
\de{Zunächst besagt Lemma \ref{\en{EN}lemelemen}, dass man für jedes gegebene symmetrische Polynom $p\in\mathbb Z[X_1,\ldots,X_n]$ eine polynomielle Darstellung $\widehat p\in\mathbb Z[Y_1,\ldots,Y_n]$ finden kann, für die gilt: $p=\widehat p (\sigma_{n;1},\sigma_{n;2},\ldots,\sigma_{n;n})$.
Dann besagt Lemma \ref{\en{EN}lemsigmaganz}, dass die $\alpha_n\cdot\sigma_{n;k}(z_1,\ldots,z_n)$ ganzzahlig sind. Also ist auch $\alpha_n^{\operatorname{deg}(\widehat p)}\cdot p(z_1,\ldots,z_n)$ ganzzahlig.

Aus der Definition der $\sigma_{n;k}$ folgt $\operatorname{deg}(\sigma_{n;k})\geq 1$.
Somit gilt $\operatorname{deg}(\widehat p)\leq\operatorname{deg}(p)$ und das Theorem ist bewiesen.}
\end{proof}

\section{\en{The Transcendence of}\de{Die Transzendenz von} \texorpdfstring{$\pi$}{\emph{pi}}}\label{\en{EN}kaptransz}
\renewcommand{\leftmark}{\en{The Transcendence of}\de{Die Transzendenz von} $\pi$}
\en{In this chapter, we give a detailed proof that $\pi$ is transcendental.
The first such proof is due to Lindemann \cite{\en{EN}Lindemann}.
We follow Hilbert's simplified proof \cite{\en{EN}Hilbert} as it was presented by Moser \cite{\en{EN}Moser}.}%
\de{In diesem Kapitel beweisen wir ausführlich die Transzendenz der Kreiszahl $\pi$.
Der erste solche Beweis stammt von Lindemann \cite{\en{EN}Lindemann}.
Wir folgen der vereinfachten Beweisidee von Hilbert~\cite{\en{EN}Hilbert} in der Darstellung von Moser \cite{\en{EN}Moser}.}

\begin{lem}\label{\en{EN}lemgregor}
\en{For all $k\in\mathbb N-\{0\}$ and $C\in\mathbb Z$ and $m\in \mathbb N$, it holds:\\
If $k$ is an integer multiple of $C$, then $C^m$ can not be divisible by $k+1$.}%
\de{Für alle $k\in\mathbb N-\{0\}$ und $C\in\mathbb Z$ und $m\in \mathbb N$ gilt:\\
Wenn $k$ ein ganzzahliges Vielfaches von $C$ ist, kann $C^m$ nicht durch $k+1$ teilbar sein.}
\end{lem}
\begin{proof}
\en{Let $k$ be an integer multiple of $C$, i.e. $k=z\cdot C$ with $z\in\mathbb Z$.}%
\de{Sei $k$ also ein ganzzahliges Vielfaches von $C$, d.h. $k=z\cdot C$ mit $z\in\mathbb Z$.}

\en{Assume that $C^m$ is divisible by $k+1$.
Then we have $C^m=v\cdot(k+1)$ with $v\in\mathbb Z$ and thus
$k^m=z^m\cdot C^m = z^m\cdot v\cdot(k+1)$. This yields:}%
\de{Angenommen, $C^m$ wäre durch $k+1$ teilbar.
Dann folgt: $C^m=v\cdot(k+1)$ mit $v\in\mathbb Z$ und somit
$k^m=z^m\cdot C^m = z^m\cdot v\cdot(k+1)$. Hieraus folgt:}
$$\frac{k^m}{k+1}=z^m\cdot v\in\mathbb Z$$
\en{This fraction can not be reduced, because $k$ and $k+1$ don't have common factors.
This yields $k+1=\pm 1$. But none of the solutions of this equations ($k=-2$ and $k=0$) is in $\mathbb N-\{0\}$, thus $C^m$ can not be divisible by $k+1$.}%
\de{Allerdings haben $k$ und $k+1$ keine gemeinsamen Faktoren, also kann der Bruch nicht gekürzt werden.
Es folgt $k+1=\pm 1$. Aber keine der beiden Lösungen dieser Gleichung ($k=-2$ und $k=0$) sind aus $\mathbb N-\{0\}$, also kann $C^m$ nicht durch $k+1$ teilbar sein.}
\end{proof}

\begin{theo}[Hilbert 1893]\label{\en{EN}thmhilbert}
\en{If $P\in\mathbb Z[X]$ is an integer polynomial of degree $n$ with zeros $s_1,\ldots,s_n\in\mathbb C$ and $a\geq 1$ is an integer, then it holds:}%
\de{Wenn $P\in\mathbb Z[X]$ ein ganzzahliges Polynom vom Grad $n$ mit Nullstellen $s_1,\ldots,s_n\in\mathbb C$ ist, dann gilt für alle natürlichen Zahlen $a\geq 1$:}
$$a+e^{s_1}+\ldots+e^{s_n} \neq 0$$
\end{theo}
\begin{proof}
\en{We give a detailed proof by contradiction in nine steps:}%
\de{Wir führen einen ausführlichen Widerspruchsbeweis in neun Schritten:}

\begin{enumerate}[wide, labelindent=0pt, itemsep=2ex,itemjoin={{asdf}}]
\renewcommand{\labelenumi}{\textbf{\en{Step}\de{Schritt}~\protect\circled{\arabic{enumi}}:}}
\renewcommand{\theenumi}{\en{Step}\de{Schritt}~\arabic{enumi}}

\item \label{\en{EN}eins}
\en{Assume the existence of a counterexample.}\de{Widerspruchsannahme.}

\dotline

\en{We assume the existence of an integer polynomial $\widehat P\in\mathbb Z[X]$ of degree $\widehat n$ with zeros $\widehat s_1,\ldots,\widehat s_{\widehat n}\in\mathbb C$ and an integer $\widehat a\geq 1$ such that}%
\de{Wir nehmen an, es gäbe ein ganzzahliges Polynom $\widehat P\in\mathbb Z[X]$ vom Grad $\widehat n$ mit Nullstellen $\widehat s_1,\ldots,\widehat s_{\widehat n}\in\mathbb C$ und eine natürliche Zahl $\widehat a\geq 1$, so dass gilt:}
$$\widehat a+e^{\widehat s_1}+\ldots+e^{\widehat s_{\widehat n}} = 0.$$

\medskip\hrule

\item \label{\en{EN}zwei}
\en{Transition to an integer $a\geq 1$ and an integer polynomial $P$ of degree $n$ with zeros $s_1,\ldots,s_n$, for which it holds}%
\de{Übergang zu einer natürlichen Zahl $a\geq 1$ und einem ganzzahligen Polynom $P$ vom Grad $n$ mit Nullstellen $s_1,\ldots,s_n$, für das gilt:}%
\begin{align}a+e^{s_1}+\ldots+e^{s_n} = 0\qquad\text{\en{and}\de{und}}\qquad P(0)\neq 0.\label{\en{EN}widsp}\end{align}

\vspace*{-1ex}\dotline

\en{If we already have $\widehat P(0)\neq 0$, use $P:=\widehat P$, $n:=\widehat n$, $a:=\widehat a$ and $s_j:=\widehat s_j$ for $j=1,\ldots,n$.}%
\de{Falls bereits $\widehat P(0)\neq 0$ ist, setze $P:=\widehat P$, $n:=\widehat n$, $a:=\widehat a$ und $s_j:=\widehat s_j$ für $j=1,\ldots,n$.}

\en{Otherwise, if $\widehat P(0)=0$ with a zero of order $v$, we use the polynomial $P(X):=\widehat P(X)/X^v$. This has only  $n:=\widehat n - v$ zeros which we denote $s_1,\ldots,s_n$ -- these form a subset of $\{\widehat s_1,\ldots,\widehat s_{\widehat n}\}$.}%
\de{Andernfalls, also wenn $\widehat P(0)=0$ gilt und dort eine $v$-fache Nullstelle vorliegt, gehen wir zum Polynom $P(X):=\widehat P(X)/X^v$ über. Dieses hat nur noch $n:=\widehat n - v$ Nullstellen, die wir $s_1,\ldots,s_n$ nennen -- eine Teilmenge der $\{\widehat s_1,\ldots,\widehat s_{\widehat n}\}$.}

\en{Then we take the assumed equation $\widehat a+e^{\widehat s_1}+\ldots+e^{\widehat s_{\widehat n}} = 0$ and use either $0$ or one of the $s_i$ for all the $\widehat s_j$ to obtain $\widehat a+e^{s_1}+\ldots+e^{ s_{n}} + v\cdot e^0 = 0$. Using $a:=\widehat a + v\geq 1$ this yields:}%
\de{Dann setzen wir in die angenommene Gleichung $\widehat a+e^{\widehat s_1}+\ldots+e^{\widehat s_{\widehat n}} = 0$ für alle $\widehat s_i$ jeweils entweder eins der $s_j$ oder $0$ ein und erhalten $\widehat a+e^{s_1}+\ldots+e^{ s_{n}} + v\cdot e^0 = 0$. Mit $a:=\widehat a + v\geq 1$ ergibt das:}
\begin{align*}a+e^{s_1}+\ldots+e^{s_n} = 0\end{align*}

\en{Thus we can replace $\widehat P$ by $P:=\widehat P/X^v$ and $\widehat a$ by $a:=\widehat a+v$ to obtain a polynomial of degree $n:=\widehat n -v$ satisfying the conditions of Thm.~\ref{\en{EN}thmhilbert}, but satisfying $P(0)\neq 0$ as well.}%
\de{Somit können wir $\widehat P$ durch $P:=\widehat P/X^v$ ersetzen und $\widehat a$ durch $a:=\widehat a+v$ und erhalten so ein Polynom vom Grad $n:=\widehat n -v$, das die Voraussetzungen von Thm.~\ref{\en{EN}thmhilbert} erfüllt aber keine Nullstelle bei $0$ hat.}

\medskip\hrule

\pagebreak

\item \label{\en{EN}drei}
\en{Define two sets $(f_k)_{k\in\mathbb N}$ and $(F_k)_{k\in\mathbb N}$ of polynomials from $\mathbb Z[X]$.}%
\de{Definiere zwei Scharen $(f_k)_{k\in\mathbb N}$ und $(F_k)_{k\in\mathbb N}$ von Polynomen aus $\mathbb Z[X]$.}

\dotline

\en{From \ref{\en{EN}zwei} we have a polynomial $P\in\mathbb Z[X]$ of degree $n$ with zeros $s_1,\ldots,s_n$ and $P(0)\neq 0$.
Denoting the leading coefficient of $P$ with $\alpha_n$ we set}%
\de{Nach \ref{\en{EN}zwei} haben wir ein Polynom $P\in\mathbb Z[X]$ vom Grad $n$ mit den Nullstellen $s_1,\ldots,s_n$ und $P(0)\neq 0$.
Der Leitkoeffizient von $P$ heiße $\alpha_n$.
Dann setzen wir}
\begin{align}\label{\en{EN}zweizwei}
    \left.
    \begin{aligned}
    f_k(X) &:= X^k \cdot \left(P(X)\right)^{k+1}\\
    \lambda_k &:= \alpha_n ^{\operatorname{deg}(f_k)} = \alpha_n^{k+n\cdot(k+1)}\\
    F_k(X) &:= \lambda_k \cdot f_k(X)
    \end{aligned}\quad\right\}\quad\text{\en{for}\de{für} }k\in\mathbb N
\end{align}

\medskip\hrule

\item \label{\en{EN}vier}
\en{Multiplication of the assumed eq.~(\ref{\en{EN}widsp}) with an integral.}%
\de{Multiplikation der angenommenen Glg.~(\ref{\en{EN}widsp}) mit einem Integral.}

\dotline

\en{First we multiply eq.~(\ref{\en{EN}widsp}) with the integral}%
\de{Zunächst multiplizieren wir Gleichung (\ref{\en{EN}widsp}) mit dem Integral}
$$\int_0^\infty F_k(x)\cdot e^{-x}\ \mathrm{d}x,$$
\en{where $F_k\in\mathbb Z[X]$ is the polynomial defined in \ref{\en{EN}drei}.
This integral runs along the real axis and converges, because the exponential function decreases faster than the polynomial $F_k$ grows.}%
\de{wobei $F_k\in\mathbb Z[X]$ eines der Polynome aus \ref{\en{EN}drei} ist.
Dieses Integral verläuft entlang der reellen Achse und es konvergiert, weil die $e$-Funktion schneller abklingt als ein Polynom $F_k$ wachsen kann.
Wir erhalten:}

$$0=a\cdot\int_0^\infty F_k(x)\cdot e^{-x}\ \mathrm{d}x + \sum_{j=1}^n e^{s_j}\cdot \int_0^\infty F_k(x)\cdot e^{-x}\ \mathrm{d}x$$

\en{Then we replace this integral $\int_0^\infty$ by an integral along the complex path from $0$ via $s$ to $\infty$, which means $\int_0^s + \int_s^\infty$ (see sketch):}%
\de{Dann ersetzen wir dieses reelle Integral $\int_0^\infty$ durch ein komplexes Kurvenintegral von $0$ über $s$ nach $\infty$, also durch $\int_0^s + \int_s^\infty$ (siehe Skizze):}

{\begin{center}
\scalebox{0.9}
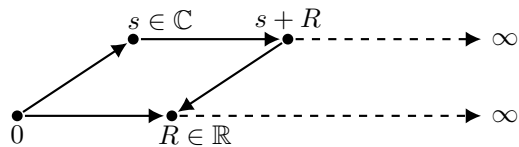
{\begin{tikzpicture}[x=10mm, y=10mm] 
   \fill (0,0) circle[radius=2pt] node[below] {$0$};
   \fill (1.5,1) circle[radius=2pt] node[above,xshift=3.5mm] {$s\in\mathbb C$};
   \draw[->, thick](0.075,0.05) -- (1.425,0.95);
   \draw[->, thick](1.6,1) -- (3.4,1);
   \fill (3.5,1) circle[radius=2pt] node[above] {$s+R$};
   \draw[->,dashed, thick] (3.6,1) -- (6,1) node[right] {$\infty$};
   \draw[->,thick](3.425,0.95) -- (2.075,0.05);
   \fill (2,0) circle[radius=2pt] node[below,xshift=3mm] {$R\in\mathbb R$};
   \draw[->,thick](0.1,0) -- (1.9,0);
   \draw[->,dashed,thick] (2.1,0) -- (6,0) node[right] {$\infty$};
\end{tikzpicture}}\end{center}}

\en{This change of path is allowed because of Cauchy's Integral Theorem and because $\int_{s+R}^R$ vanishes for real $R\rightarrow\infty$.
We obtain:}%
\de{Diesen Wechsel des Integrationswegs erlaubt uns der Cauchy'sche Integralsatz, weil $\int_{s+R}^R$ für reelle $R\rightarrow\infty$ verschwindet.
Es folgt:}
\begin{align}0=\underbrace{a\cdot\int\limits_0^\infty F_k(x)\ e^{-x}\ \mathrm{d}x}_{=:Q_k}
+ \underbrace{\sum_{j=1}^n e^{s_j} \int\limits_{s_j}^\infty F_k(x)\ e^{-x}\ \mathrm{d}x}_{=:R_k}
+ \underbrace{\sum_{j=1}^n e^{s_j} \int\limits_0^{s_j} F_k(x)\ e^{-x}\ \mathrm{d}x}_{=:S_k}\label{\en{EN}defQRS}\end{align}

\en{As depicted in the sketch above, the integration path in $Q_k$ runs along the real axis, in $R_k$ it runs parallel to the real axis, and in $Q_k$ it's the line segment from $0$ to $s_j$.
But actually, because of Cauchy's Integral theorem, the exact path is not important.}%
\de{Wie es auch in der Skizze dargestellt ist, wird bei $Q_k$ entlang der reellen Achse integriert, bei $R_k$ wird von $s_j$ aus parallel zur reellen Achse integriert und bei $S_k$ auf direktem Weg von $0$ nach $s_j$.
Letztlich spielt der genaue Weg aber wegen des Cauchy'schen Integralsatzes gar keine Rolle.}

\en{In the following steps we will show that there are some $k\in\mathbb N$ with}%
\de{Im Folgenden werden wir zeigen, dass es gewisse $k\in\mathbb N$ gibt, für die gilt:}
$$\frac{Q_k+R_k}{k!}\in\mathbb Z\qquad\text{\en{and}\de{und}}\qquad \frac{Q_k+R_k}{k!}\neq 0\qquad\text{\en{and}\de{und}}\qquad \left|\frac{S_k}{k!}\right|<1$$
\en{For these $k$, the term $(Q_k+R_k+S_k)/k!$ can't be zero. This will yield the desired contradiction to $0=Q_k+R_k+S_k$ in \ref{\en{EN}neun}.}%
\de{Für diese $k$ kann also $(Q_k+R_k+S_k)/k!$ nicht Null sein. Das liefert dann in \ref{\en{EN}neun} den gewünschten Widerspruch zu $0=Q_k+R_k+S_k$.}
\medskip\hrule

\pagebreak

\item \label{\en{EN}fuenf}
\en{For all $k\in\mathbb N$ it holds: $R_k$ is an integer and divisible by $(k+1)!$.}%
\de{$R_k$ ist für alle $k\in\mathbb N$ eine durch $(k+1)!$ teilbare ganze Zahl.}

\dotline

\en{To evaluate $R_k$, we parameterize the integration paths by $z_j(t):=s_j+t$:}%
\de{Um die einzelnen Summanden in $R_k$ zu berechnen, setzen wir jeweils die Para\-me\-tri\-sierung $z_j(t):=s_j+t$ des Integrationswegs ein:}
$$R_k := \sum_{j=1}^n e^{s_j} \int\limits_{s_j}^\infty \overbrace{\lambda_k\cdot f_k(x)}^{F_k(x)}\cdot e^{-x}\ \mathrm{d}x = \sum_{j=1}^n \int\limits_{0}^\infty \lambda_k\cdot f_k(t+s_j)\cdot e^{-t}\ \mathrm{d}t $$
\en{If we take the polynomial $f_k(t+s_j)$ as a polynomial in \emph{one} variable $t$, it doesn't have integer coefficients any more. But if we take $f_k(X+Y)$ as a polynomial in \emph{two} variables and sort by powers of $X$ we obtain}%
\de{Das Polynom $f_k(t+s_j)$ hat als Polynom in der \emph{einen} Variable $t$ keine ganzzahligen Koeffizienten mehr, aber wenn wir $f_k(X+Y)$ als Polynom in \emph{zwei} Variablen schreiben und nach Potenzen von $X$ sortieren, erhalten wir}
\begin{align*}
f_k(t+s_j) &= \sum_{l=0}^d H_l(s_j)\cdot t^l
\end{align*}
\en{with \emph{integer} polynomials $H_l\in\mathbb Z[Y]$ whose degrees satisfy $\operatorname{deg}(H_l)\leq\operatorname{deg}(f_k)$.}%
\de{mit \emph{ganzzahligen} Polynomen $H_l\in\mathbb Z[Y]$, für deren Grade  $\operatorname{deg}(H_l)\leq\operatorname{deg}(f_k)$ gilt.}

\en{From its definition in \ref{\en{EN}drei}, $f_k(x)$ has zeros of order $k+1$ in each $s_j$, thus $f_k(t+s_j)$ has zeros of order $k+1$ at $t=0$ for each $s_j$.
This proves that $H_l(s_j)$ vanishes for all $l\leq k$ and that the $l$-summation can start with $l=k+1$:}%
\de{Nach der Definition in \ref{\en{EN}drei} hat $f_k(x)$ in jedem $s_j$ eine $k+1$-fache Nullstelle, also hat $f_k(t+s_j)$ für alle $s_j$ bei $t=0$ eine $k+1$-fache Nullstelle.
Hieraus folgt, dass $H_l(s_j)$ für alle $l\leq k$ verschwindet und die Summation über $l$ erst bei $l=k+1$ beginnen muss:}
$$f_k(t+s_j)=\sum_{l=k+1}^d H_l(s_j)\cdot t^l\qquad\text{\en{with}\de{mit} }H_l\in\mathbb Z[Y].$$
\en{This yields:}\de{Daraus folgt:}
\begin{align*}
R_k &= \sum_{j=1}^n \int_{0}^\infty \lambda_k\cdot\sum_{l=k+1}^d H_l(s_j)\cdot t^l\cdot e^{-t}\ \mathrm{d}t
= \lambda_k\cdot\sum_{l=k+1}^d \underbrace{\sum_{j=1}^n  H_l(s_j)}_{=:c_l} \cdot \underbrace{\int_{0}^\infty t^l\cdot e^{-t}\ \mathrm{d}t}_{=:d_l}
\end{align*}

\en{
Here we notice:
$c_l:=\sum_{j=1}^n  H_l(s_j)$ is a symmetric polynomial expression in the zeros $s_j$ of the integer polynomial $P$. This polynomial $P$ has the leading coefficient $\alpha_n$.
Thus the theorem of symmetric polynomials (Thm.~\ref{\en{EN}theosymm}) tells that}%
\de{Um Thm.~\ref{\en{EN}theosymm} über symmetrische Polynome auf $c_l:=\sum_{j=1}^n  H_l(s_j)$ anzuwenden bemerken wir:
$c_l$ ist ein symmetrischer polynomieller Ausdruck in den Nullstellen $s_j$ des ganzzahligen Polynoms $P$. Dieses Polynom $P$ hat den Leitkoeffizient $\alpha_n$.
Thm.~\ref{\en{EN}theosymm} besagt also, dass}
$\lambda_k\cdot c_l = \alpha_n^{\operatorname{deg}(f_k)}\cdot\sum_{j=1}^n  H_l(s_j)$
\en{is integral, since $\operatorname{deg}(H_l)\leq \operatorname{deg}(f_k)$.}%
\de{ganzzahlig ist, weil $\operatorname{deg}(H_l)\leq \operatorname{deg}(f_k)$.}

\en{Next we prove by induction and partial integration that $d_l=l!$:}%
\de{Weiter ist auch $d_l=l!$ für alle $l$ ganzzahlig, was wir per vollständiger Induktion und partieller Integration beweisen:}
\begin{align*}
    d_0 &= \int_0^\infty e^{-t}\ \mathrm{d}t=\bigl[-e^{-t}\bigr]_0^\infty = 0 - (-1) = 1 = 0!\\
    d_{l+1} &= \int_0^\infty t^{l+1}\cdot e^{-t}\ \mathrm{d}t = \bigl[t^{l+1}\cdot(-e^{-t})\bigr]_0^\infty - \int_0^\infty (l+1)\cdot t^{l}\cdot (-e^{-t})\ \mathrm{d}t\\
    &= (l+1) \cdot \int_0^\infty t^{l}\cdot e^{-t}\ \mathrm{d}t = (l+1)\cdot d_l = (l+1)\cdot l! = (l+1)!
\end{align*}
\en{Thus we have proven that $R_k$ is an integer:}%
\de{Insgesamt haben wir also bewiesen, dass $R_k$ ganzzahlig ist:}
$$R_k = \sum_{l=k+1}^d \underbrace{\lambda_k\cdot c_l}_{\in\mathbb Z}\cdot \underbrace{d_l}_{=l!}\quad\Longrightarrow\quad R_k\in\mathbb Z.$$ 
\en{It remains to prove that $R_k$ is divisible by $(k+1)!$:
Since the sum starts only at $l=k+1$, all $d_l=l!$ in the sum are divisible by $(k+1)!$, thus $R_k$ is divisible by $(k+1)!$ as well.}%
\de{Es fehlt noch zu zeigen, dass $R_k$ durch $(k+1)!$ teilbar ist.
Das gilt aber, weil die Summe erst bei $l=k+1$ beginnt und somit alle summierten $d_l=l!$ durch $(k+1)!$ teilbar sind.}

\medskip\hrule

\pagebreak

\item \label{\en{EN}sechs}
\en{Every $Q_k$ is an integer divisible by $k!$ and it holds}%
\de{$Q_k$ ist für alle $k\in\mathbb N$ eine durch $k!$ teilbare ganze Zahl und es gilt}
$$\frac{Q_k}{k!}\equiv  a\cdot\lambda_k\cdot P(0)^{k+1}\quad(\operatorname{mod} k+1).$$

\dotline

\en{The definition of $Q_k$ in eq.~(\ref{\en{EN}defQRS}) (see \ref{\en{EN}vier}) is:}%
\de{$Q_k$ wurde in Glg.~(\ref{\en{EN}defQRS}) (siehe \ref{\en{EN}vier}) wie folgt definiert:}
\begin{align*}
    Q_k &:= a\cdot\int\limits_0^\infty F_k(x)\cdot e^{-x}\ \mathrm{d}x
        =  a\cdot\int\limits_0^\infty  \lambda_k\cdot \underbrace{x^{k}\cdot (P(x))^{k+1}}_{=f_k(x)}\cdot e^{-x}\ \mathrm{d}x
\end{align*}
\en{Here, $f_k$ is an integer polynomial with a zero of order $k$ at $x=0$ (because we forced $P(0)\neq 0$ in \ref{\en{EN}drei}). This yields}%
\de{Hier ist $f_k$ ein ganzzahliges Polynom, das eine $k$-fache Nullstelle bei $x=0$ hat (beachte, dass wir in \ref{\en{EN}drei} $P(0)\neq 0$ erzwungen haben). Also gilt}
\begin{align}
  f_k(x):=x^{k}\cdot (P(x))^{k+1}=\sum_{l=k}^{l_{\operatorname{max}}} \beta_l \cdot x^l\label{\en{EN}defibeta}
\end{align}
\en{with integer coefficients $\beta_l$ and thus}%
\de{mit ganzzahligen Koeffizienten $\beta_l$ und somit}
\begin{align*}
    Q_k &=  a\cdot\int\limits_0^\infty  \lambda_k\cdot \sum_{l=k}^{l_{\operatorname{max}}} \beta_l \cdot x^l\cdot e^{-x}\ \mathrm{d}x\\
    &=  a\cdot\lambda_k\cdot \sum_{l=k}^{l_{\operatorname{max}}} \beta_l \cdot\int_0^\infty  x^l\cdot e^{-x}\ \mathrm{d}x
    = a\cdot\lambda_k\cdot \sum_{l=k}^{l_{\operatorname{max}}} \beta_l \cdot l!\\
\Longrightarrow\quad\frac{Q_k}{k!} &= a\cdot\lambda_k\cdot \sum_{l=k}^{l_{\operatorname{max}}} \beta_l \cdot \frac{l!}{k!}
\end{align*}
\en{Every $l!$ in this summation is divisible by $k!$, because we sum only those $l$ with $l\geq k$.
Modulo $k+1$, all summands with $l\geq k+1$ vanish (because the terms $l!/k!$ are divisible by $k+1$ for $l\geq k+1$) and it holds:}%
\de{Weil die Summe erst bei $l=k$ beginnt, sind alle summierten $l!$ durch $k!$ teilbar.
Modulo $k+1$ fallen alle Summanden ab dem $k+1$-ten weg (weil ab dann alle $l!/k!$ durch $k+1$ teilbar sind) und es gilt:}
$$\frac{Q_k}{k!}\equiv a\cdot \lambda_k\cdot \beta_k\quad(\operatorname{mod} k+1)$$
\en{From the definition of $\beta_l$ in eq. (\ref{\en{EN}defibeta}) we deduce $\beta_k=P(0)^{k+1}$ and thus:}%
\de{Aus der Definition der $\beta_l$ in Glg. (\ref{\en{EN}defibeta}) folgt $\beta_k=P(0)^{k+1}$ und somit:} 
$$\frac{Q_k}{k!}\equiv a\cdot \lambda_k\cdot P(0)^{k+1}\quad(\operatorname{mod} k+1)$$

\medskip\hrule

\item \label{\en{EN}sieben}
\en{If $k\neq 0$ is a multiple of $a\cdot\alpha_n\cdot P(0)$, then $Q_k+R_k$ is divisible by $k!$, but the integer $(Q_k+R_k)/k!$ is nonzero.}%
\de{Falls $k\neq 0$ ein Vielfaches von $a\cdot\alpha_n\cdot P(0)$ ist, ist $Q_k+R_k$ durch $k!$ teilbar, aber die ganze Zahl $(Q_k+R_k)/k!$ ist nicht Null.}

\dotline

\en{We already proved the integrality of $(Q_k+R_k)/k!$ in \ref{\en{EN}fuenf} and \ref{\en{EN}sechs}.}%
\de{Die Ganzzahligkeit von $(Q_k+R_k)/k!$ haben wir in \ref{\en{EN}fuenf} und \ref{\en{EN}sechs} bereits bewiesen.}

\en{We also proved in \ref{\en{EN}fuenf} that $R_k/(k+1)!$ is an integer. Denoting this integer $r_k$ yields $R_k/k!=(k+1)\cdot r_k$ and thus}%
\de{Außerdem haben wir in \ref{\en{EN}fuenf} bewiesen, dass auch $R_k/(k+1)!$ ganzzahlig ist. Wenn wir diese ganze Zahl $r_k$ nennen, folgt $R_k/k!=(k+1)\cdot r_k$ und insbesondere}
$$R_k/k!\equiv 0 \quad (\operatorname{mod} k+1).$$
\en{Combining this with the statement of \ref{\en{EN}sechs} yields}%
\de{Zusammen mit der Aussage von \ref{\en{EN}sechs} folgt}
$$\frac{Q_k+R_k}{k!}\equiv a\cdot \lambda_k\cdot P(0)^{k+1} \quad (\operatorname{mod} k+1),$$
\en{where $\lambda_k=\alpha_n^{k+nk+n}$ (see (\ref{\en{EN}zweizwei})). From \ref{\en{EN}drei} we know $P(0)\neq 0$ and thus}%
\de{mit $\lambda_k=\alpha_n^{k+nk+n}$ (siehe (\ref{\en{EN}zweizwei})). Nach \ref{\en{EN}drei} gilt $P(0)\neq 0$ und somit}
$$C:=a\cdot\alpha_n\cdot P(0)\neq 0.$$
\en{Lemma \ref{\en{EN}lemgregor} tells us for $m:=k+nk+n+1$: If $k\neq 0$ is a multiple of $C$, then}%
\de{Aus Lemma \ref{\en{EN}lemgregor} folgt mit $m:=k+nk+n+1$: Wenn $k\neq 0$ ein Vielfaches von $C$ ist, kann}
$$C^m=C^{k+nk+n+1}=a^{k+nk+n+1}\cdot\alpha_n^{k+nk+n+1}\cdot P(0)^{k+nk+n+1}$$
\en{can't be divisible by $k+1$. All the more, $a\cdot\alpha_n^{k+nk+n}\cdot P(0)^{k+1}$ can't be divisible by $k+1$ and we obtain:}%
\de{nicht durch $k+1$ teilbar sein. Erst recht kann dann auch $a\cdot\alpha_n^{k+nk+n}\cdot P(0)^{k+1}$ nicht durch $k+1$ teilbar sein und es folgt:}
$$\frac{Q_k+R_k}{k!} \equiv a\cdot \alpha_n^{k+nk+n}\cdot P(0)^{k+1}\not\equiv 0 \quad (\operatorname{mod} k+1).$$
\en{In particular, $(Q_k+R_k)/k!$ doesn't vanish for these $k$.}%
\de{Insbesondere kann $(Q_k+R_k)/k!$ für diese $k$ nicht Null sein.}

\medskip\hrule

\item \label{\en{EN}acht}
\en{For all sufficiently large $k$ it holds:}%
\de{Für hinreichend große $k$ gilt:}
$$\left|\frac{S_k}{k!}\right|<1.$$

\dotline

\en{The definition of $S_k$ from eq. (\ref{\en{EN}defQRS}) in \ref{\en{EN}vier} reads:}%
\de{Die Definition der $S_k$ aus Glg. (\ref{\en{EN}defQRS}) in \ref{\en{EN}vier} lautet:}
$$S_k := \sum_{j=1}^n e^{s_j} \int_0^{s_j} F_k(x)\cdot e^{-x}\ \mathrm{d}x.$$
\en{Here we use the definition of $F_k$ from \ref{\en{EN}drei} to obtain:}%
\de{Hier setzen wir noch die Definition der $F_k$ aus \ref{\en{EN}drei} ein:}
\begin{align*}
    S_k &= \sum_{j=1}^n e^{s_j} \int_0^{s_j} \lambda_k\cdot x^k \cdot (P(x))^{k+1}\cdot e^{-x}\ \mathrm{d}x\\
    &= \lambda_k\cdot \sum_{j=1}^n e^{s_j} \int_0^{s_j} (x\cdot P(x))^{k}\cdot (P(x)\cdot e^{-x})\ \mathrm{d}x
\end{align*}
\en{Somewhere on the combined integeration paths between $0$ and the $s_j$, there is a maximal absolute value of $x\cdot P(x)$ which we denote with $M$.
We also denote the maximal absolute value of $P(x)\cdot e^{-x}$ on these paths with $m$. This yields the confining inequality}%
\de{Dann gibt es irgendwo auf den $n$ Integrationswegen zwischen $0$ und $s_j$ ein Betragsmaximum von $x\cdot P(x)$, das wir mit $M$ bezeichnen. Ebenso bezeichnen wir das Betragsmaximum von $P(x)\cdot e^{-x}$ mit $m$. Also gilt die Abschätzung}
\begin{align*}
    |S_k| &\leq |\lambda_k|\cdot \sum_{j=1}^n |e^{s_j}| \cdot |s_j|\cdot M^{k}\cdot m.
\end{align*}
\en{Here we use $|\lambda_k|=|\alpha_n|^{k+nk+n}\leq|\alpha_n|^{3nk}$ and write the confining inequality in the form $|S_k|\leq U\cdot V^k$ with $U:= m\cdot \sum_{j=1}^n |e^{s_j}| \cdot |s_j|$ and $V:=|\alpha_n|^{3n}\cdot M$.}%
\de{Hier nutzen wir noch $|\lambda_k|=|\alpha_n|^{k+nk+n}\leq|\alpha_n|^{3nk}$ und schreiben die Abschätzung in der Form $|S_k|\leq U\cdot V^k$ mit $U:= m\cdot \sum_{j=1}^n |e^{s_j}| \cdot |s_j|$ und $V:=|\alpha_n|^{3n}\cdot M$.}

\en{The ratio test tells that $|S_k/k!| \leq U\cdot V^k/k!$ tends to zero while $k\rightarrow\infty$ and we have proven \ref{\en{EN}acht}.}%
\de{Mit dem Quotientenkriterium folgt, dass $|S_k/k!| \leq U\cdot V^k/k!$ eine Nullfolge ist, also ist auch \ref{\en{EN}acht} bewiesen.}
\medskip\hrule

\item \label{\en{EN}neun}
\en{Derivation of the contradiction.}%
\de{Herleitung des Widerspruchs.}

\dotline

\en{For all $k$ being sufficiently large multiples of $a\cdot\alpha_n\cdot P(0)$ we proved:}%
\de{Wir haben gezeigt, dass für alle hinreichend großen $k$, die Vielfache von $a\cdot\alpha_n\cdot P(0)$ sind, gilt:}
$$\frac{Q_k+R_k}{k!}\in\mathbb Z\qquad\text{\en{and}\de{und}}\qquad \frac{Q_k+R_k}{k!}\neq 0\qquad\text{\en{and}\de{und}}\qquad \left|\frac{S_k}{k!}\right|<1$$
\en{Such $k$ exist, because we have $a\cdot \alpha_n\cdot P(0)\neq 0$ (see \ref{\en{EN}zwei}).
For these $k$ it holds:}%
\de{Solche $k$ gibt es, weil $a\cdot \alpha_n\cdot P(0)\neq 0$ ist (siehe \ref{\en{EN}zwei}).
Für diese $k$ gilt:}
$$\left|\frac{Q_k+R_k+S_k}{k!}\right|\geq\underbrace{\left|\frac{Q_k+R_k}{k!}\right|}_{\geq 1}-\underbrace{\left|\frac{S_k}{k!}\right|}_{<1}>0$$
\en{Thus the term $(Q_k+R_k+S_k)/k!$ must be nonzero for these $k$, which yields}%
\de{Für diese $k$ kann $(Q_k+R_k+S_k)/k!$ also nicht Null sein und es folgt}
$$Q_k+R_k+S_k\neq 0.$$
\en{But eq.~(\ref{\en{EN}defQRS}) in \ref{\en{EN}vier} tells $Q_k+R_k+S_k=0$ for \emph{all} integer $k$, which yields a contradiction.
Thus the assumption from \ref{\en{EN}eins} is wrong -- and we have proven Thm.~\ref{\en{EN}thmhilbert}.}%
\de{Aber Glg.~(\ref{\en{EN}defQRS}) in \ref{\en{EN}vier} besagt, dass $Q_k+R_k+S_k=0$ für \emph{alle} natürlichen $k$ gilt.\linebreak Aus diesem Widerspruch folgt, dass die Annahme aus \ref{\en{EN}eins} falsch gewesen sein muss und Hilberts Theorem \ref{\en{EN}thmhilbert} ist bewiesen.}\qedhere
\end{enumerate}
\end{proof}

\bigskip
\begin{theo}\label{\en{EN}pitransz}
\vspace*{2mm}\en{The number $\pi$ is transcendental.}\de{Die Kreiszahl $\pi$ ist transzendent.}\vspace*{3mm}
\end{theo}
\begin{proof}
\en{This theorem is a consequence of Hilbert's Thm.~\ref{\en{EN}thmhilbert}:}%
\de{Dieses Theorem ist eine Folge aus Hilberts Theorem \ref{\en{EN}thmhilbert}:}
\en{Again we do a proof by contradiction and assume $\pi$ to be algebraic. 
But then $x_1:=i\pi$ is algebraic as well, and there is a nontrivial polynomial $Q\in\mathbb Z[X]$ of degree $m$ with $Q(x_1)=0$.
Next we denote the other zeros of $Q$ with $x_2,\ldots,x_m$.}%
\de{Auch hier führen wir einen Widerspruchsbeweis und nehmen an, dass $\pi$ algebraisch ist. 
Dann ist auch $x_1:=i\pi$ algebraisch, also gibt es ein nichttriviales Polynom $Q\in\mathbb Z[X]$ vom Grad $m$ mit $Q(x_1)=0$.
Weiter bezeichnen wir die übrigen Nullstellen von $Q$ mit $x_2,\ldots,x_m$.}

\en{But $e^{x_1}=e^{i\pi}=-1$ yields $\prod_{i=1}^m (1+e^{x_i}) = 0$. When expanding this product, we use the functional equation of the exponential function and obtain}%
\de{Aus $e^{x_1}=e^{i\pi}=-1$ folgt $\prod_{i=1}^m (1+e^{x_i}) = 0$. Wenn wir dieses Produkt ausmultiplizieren, können wir die Funktionalgleichung der $e$-Funktion anwenden und erhalten}
\begin{align}
  \prod_{i=1}^m (1+e^{x_i}) = 0 = 1 + e^{s_1} + \ldots + e^{s_n}\en{.}\de{,}\label{\en{EN}widersp}
\end{align}
\en{Here we have $2^m = 1+n$ summands and each $s_j$ is the sum of some $x_i$.
Using these $s_j$ and the leading coefficient $\alpha_m$ of $Q$ we define the polynomial}%
\de{wobei insgesamt $2^m = 1+n$ Summanden entstanden sind und jedes $s_j$ die Summe einiger $x_i$ ist.
Mit diesen $s_j$ und dem Leitkoeffizient $\alpha_m$ von $Q$ bilden wir das Polynom}
$$P:=\alpha_m^n\cdot(X-s_1)\cdot \ldots\cdot(X-s_n)\in\mathbb C[X]$$
\en{Applying Vieta's formula (Lemma \ref{\en{EN}lemvieta}) yields}%
\de{Der Satz von Vieta (Lemma \ref{\en{EN}lemvieta}) liefert}
$$P=\alpha_m^n\cdot X^n + \alpha_m^n\cdot\sum_{k=1}^{n}(-1)^k\cdot\sigma_{n;k}(s_1,\ldots,s_n)\cdot X^{n-k}$$
\en{with the elementary symmetric polynomials $\sigma_{n;k}$ from Def.~\ref{\en{EN}defielem}.}%
\de{mit den elementarsymmetrischen Polynomen $\sigma_{n;k}$ aus Def.~\ref{\en{EN}defielem}.}

\en{But since the $s_j$ contain all possible sums of the $x_i$, the terms $\sigma_{n;k}(s_1,\ldots,s_n)$ can be written as \emph{symmetric} expressions $v_{k}(x_1,\ldots,x_m)$:}%
\de{Weil aber unter den $s_j$ alle denkbaren Summen der $x_i$ vorkommen, kann man die $\sigma_{n;k}(s_1,\ldots,s_n)$ auch als \emph{symmetrische} Ausdrücke $v_{k}(x_1,\ldots,x_m)$ schreiben:}
$$P=\alpha_m^n\cdot X^n + \alpha_m^n\cdot \sum_{k=1}^{n}(-1)^k\cdot v_k(x_1,\ldots,x_m)\cdot X^{n-k}$$

\en{The polynomial $Q$ whose existence we assumed in the beginning is of degree $m$, has leading coefficient $\alpha_m$ and the zeros $x_1,\ldots,x_m$. Thus Thm.~\ref{\en{EN}theosymm} proves the integrality of the $\alpha_m^{\operatorname{deg}(v_k)}\cdot v_k(x_1,\ldots,x_m)$.}%
\de{Das anfangs angenommene ganzzahlige Polynom $Q$ hat den Grad $m$, den Leitkoeffizient $\alpha_m$ und die Nullstellen $x_1,\ldots,x_m$. Aus Thm.~\ref{\en{EN}theosymm} folgt daher die Ganzzahligkeit der $\alpha_m^{\operatorname{deg}(v_k)}\cdot v_k(x_1,\ldots,x_m)$.}
\en{But since $\operatorname{deg}(v_k)=\operatorname{deg}(\sigma_{n;k})\leq n$, this shows that $P$ is an \emph{integer} polynomial of degree $n$ with zeros $s_1,\ldots,s_n\in\mathbb C$. Applying Hilbert's Thm.~\ref{\en{EN}thmhilbert} to $P$ yields:}%
\de{Wegen $\operatorname{deg}(v_k)=\operatorname{deg}(\sigma_{n;k})\leq n$ folgt, dass $P$ ein \emph{ganzzahliges} Polynom vom Grad $n$ mit Nullstellen $s_1,\ldots,s_n\in\mathbb C$ ist.
Deshalb dürfen wir Hilberts Thm.~\ref{\en{EN}thmhilbert} auf $P$ anwenden und erhalten:}
$$1+e^{s_1}+\ldots+e^{s_n} \neq 0$$
\en{This is a contradiction to eq.~(\ref{\en{EN}widersp}), thus $\pi$ can not be algebraic.}%
\de{Das ist ein Widerspruch zu Gleichung (\ref{\en{EN}widersp}), also kann $\pi$ nicht algebraisch sein.}
\end{proof}

\vspace*{1.25mm}
\section{\en{Algebraic Numbers}\de{Algebraische Zahlen}}\label{\en{EN}anhalg}
\renewcommand{\leftmark}{\en{Algebraic Numbers}\de{Algebraische Zahlen}}
\en{In this chapter, we prove that the set of algebraic numbers is closed under addition, subtraction, multiplication, division, and the extraction of square roots.
The proof follows Timothy Gowers \cite{\en{EN}Gowers} and uses only basic linear algebra.}%
\de{In diesem Kapitel beweisen wir, dass die Menge der algebraischen Zahlen abgeschlossen unter Addition, Subtraktion, Multiplikation, Division und Quadratwurzelziehen ist.
Der Beweis folgt Timothy Gowers \cite{\en{EN}Gowers} und verwendet nur grundlegende lineare Algebra.}

\begin{defi}
\en{A number $ x $ is called algebraic of degree $ n $ if there is a polynomial of degree $ n $ with rational coefficients that has $ x $ as the zero, i.e. $ x ^ n + \sum_ {k = 0} ^ {n -1} \alpha_k \cdot x ^ k = 0 $ with $ \alpha_k \in \mathbb Q $, and there is no such polynomial of lower degree.}%
\de{Eine Zahl $x$ heißt algebraisch vom Grad $n$, wenn es ein Polynom vom Grad $n$ mit rationalen Koeffizienten gibt, das $x$ als Nullstelle hat, also $x^n+\sum_{k=0}^{n-1}\alpha_k\cdot x^k=0$ mit $\alpha_k\in\mathbb Q$, und es kein solches Polynom mit niedrigerem Grad gibt.}
\end{defi}

\begin{thm}\label{\en{EN}satzalgabgeschl}
\en{The set of algebraic numbers is closed under addition, subtraction, multiplication, division, and the extraction of square roots.}%
\de{Die Menge der algebraischen Zahlen ist abgeschlossen unter Addition, Subtraktion, Multiplikation, Division und Quadratwurzelziehen.}
\end{thm}
\begin{proof}
\en{Closure with respect to}\de{Die Abgeschlossenheit bezüglich} \ldots
\begin{itemize}
    \en{\item \ldots\ addition is proven in Lemma \ref{\en{EN}lemsum},
    \item \ldots\ subtraction is proven in Lemma \ref{\en{EN}lemsum} and \ref{\en{EN}lemwurz} because $a-b=a+(-b)$,
    \item \ldots\ multiplication is proven in Lemma \ref{\en{EN}lemmult},
    \item \ldots\ division is proven in Lemma \ref{\en{EN}lemmult} and \ref{\en{EN}lemwurz} because $a/b=a\cdot 1/b$,
    \item \ldots\ the extraction of square roots is proven in Lemma \ref{\en{EN}lemwurz}.\qedhere}%
    \de{\item \ldots\ Addition folgt aus Lemma \ref{\en{EN}lemsum},
    \item \ldots\ Subtraktion folgt wegen $a-b=a+(-b)$ aus Lemma \ref{\en{EN}lemsum} und \ref{\en{EN}lemwurz},
    \item \ldots\ Multiplikation folgt aus Lemma \ref{\en{EN}lemmult},
    \item \ldots\ Division folgt wegen $a/b=a\cdot 1/b$ aus Lemma \ref{\en{EN}lemmult} und \ref{\en{EN}lemwurz},
    \item \ldots\ Quadratwurzelziehen folgt aus Lemma \ref{\en{EN}lemwurz}.\qedhere}
\end{itemize}
\end{proof}

\begin{lem}\label{\en{EN}lemsum}
\en{If $a$ and $b$ are algebraic, then $a+b$ is also algebraic.}%
\de{Wenn $a$ und $b$ algebraisch sind, dann ist auch $a+b$ algebraisch.}
\end{lem}
\begin{proof}
\en{Let $n$ be the degree of $a$ and $m$ the degree of $b$.
Then there are rational coefficients $ \alpha_k $ and $ \beta_k $ with $a^n+\sum_{k=0}^{n-1}\alpha_k\cdot a^k=0$ and $b^m+\sum_{k=0}^{m-1}\beta_k\cdot b^k=0$.}%
\de{Sei $n$ der Grad von $a$ und $m$ der Grad von $b$.
Dann gibt es rationale Koeffizienten $\alpha_k$ und $\beta_k$ mit $a^n+\sum_{k=0}^{n-1}\alpha_k\cdot a^k=0$ und $b^m+\sum_{k=0}^{m-1}\beta_k\cdot b^k=0$.}

\en{Now consider the sequence of the powers $ (a + b) ^ k $ for $ k \in \mathbb N $. Using the binomial theorem we can write each of these powers as an integer linear combination of expressions of the form $ a ^ r \cdot b ^ s $.

Whenever a power $ a ^ r $ with $ r \geq n $ occurs, one can replace $ a ^ n $ with $ - \sum_ {k = 0} ^ {n-1} \alpha_k \cdot a ^ k $ , and the same with $ b ^ m $.

In this way and for every $ k \in \mathbb N $ we can find a linear combination}%
\de{Betrachte nun die Folge der Potenzen $(a+b)^k$ für $k\in\mathbb N$. Mit dem binomischen Lehrsatz kann man jede dieser Potenzen als Summe ganzzahliger Vielfacher von Ausdrücken der Form $a^r\cdot b^s$ schreiben.

Immer wenn eine Potenz $a^r$ mit $r\geq n$ auftritt, kann man $a^n$ durch $-\sum_{k=0}^{n-1}\alpha_k\cdot a^k$ ersetzen, genauso bei $b^m$.

Auf diese Art kann man für jedes $k\in\mathbb N$ eine Linearkombination}
$$(a+b)^k = \sum_{\substack{r,s\in\mathbb N\\r+s=k}} \binom{r+s}{s}\cdot a^r\cdot b^s
          = \sum_{\substack{0\leq r\leq n-1\\0\leq s\leq m-1}}\gamma_{r;s}\cdot a^r\cdot b^s$$
\en{with rational coefficients $ \gamma_ {r; s} $.

This illustrates that the powers $ (a + b) ^ k $ with $ 0 \leq k \leq m \cdot n $ can all be written as rational linear combinations of $ a ^ r \cdot b ^ s $ with $ 0 \leq r <n $ and $ 0 \leq s <m $.

Consequently, these $ m \cdot n + 1 $ powers all lie in a vector space of dimension $\leq m \cdot n $, so they must be linearly dependent and there is a representation $ \sum_ {k = 0} ^ { m \cdot n} \delta_k \cdot (a + b) ^ k = 0 $ with rational coefficients $ \delta_k \in \mathbb Q $.

Here we observe that $ a + b $ is a root of the polynomial $ P = \sum_ {k = 0} ^ {m \cdot n} \delta_k \cdot X ^ k \in \mathbb Z [X] $, i.e.~that $ a + b $ is algebraic of degree $ \leq m \cdot n $.}%
\de{mit rationalen Koeffizienten $\gamma_{r;s}$ finden.

Die Potenzen $(a+b)^k$ mit $0\leq k\leq m\cdot n$ lassen sich also alle als rationale Linearkombinationen der $a^r\cdot b^s$ mit $0\leq r<n$ und $0\leq s<m$ darstellen.

Folglich liegen diese $m\cdot n +1$ Potenzen alle in einem höchstens $m\cdot n$-dimen\-sionalen Vektorraum, sie müssen also linear abhängig sein und gibt es eine Darstellung $\sum_{k=0}^{m\cdot n}\delta_k\cdot (a+b)^k=0$ mit rationalen Koeffizienten $\delta_k\in\mathbb Q$.

Hier erkennen wir, dass $a+b$ eine Nullstelle des Polynoms $P=\sum_{k=0}^{m\cdot n}\delta_k\cdot X^k\in\mathbb Z[X]$ ist, d.h. dass $a+b$ algebraisch vom Grad $\leq m\cdot n$ ist.}
\end{proof}

\begin{lem}\label{\en{EN}lemmult}
\en{If $a$ and $b$ are algebraic, then $a\cdot b$ is also algebraic.}%
\de{Wenn $a$ und $b$ algebraisch sind, dann ist auch $a\cdot b$ algebraisch.}
\end{lem}
\begin{proof}
\en{The proof is similar to the one for $ a + b $, except that the step with the binomial theorem is omitted:

Let $ n $ be the degree of $ a $ and $ m $ the degree of $ b $.
Then there are rational coefficients $ \alpha_k $ and $ \beta_k $ with $ a ^ n + \sum_ {k = 0} ^ {n-1} \alpha_k \cdot a ^ k = 0 $ and $ b ^ m + \sum_ { k = 0} ^ {m-1} \beta_k \cdot b ^ k = 0 $.

Now consider the sequence of the powers $ (a \cdot b) ^ k $ for $ k \in \mathbb N $. Whenever a power $ a ^ r $ with $ r \geq n $ occurs, you can replace $ a ^ n $ with $ - \sum_ {k = 0} ^ {n-1} \alpha_k \cdot a ^ k $ , and the same with $ b ^ m $.

The powers $ (a \cdot b) ^ k $ with $ 0 \leq k \leq m \cdot n $ can therefore all be written as rational linear combinations of $ a ^ r \cdot b ^ s $ with $ 0 \leq r <n $ and $ 0 \leq s <m $.

Consequently, these $ m \cdot n + 1 $ powers all lie in a vector space of dimension $\leq m \cdot n $, so they must be linearly dependent and there is a representation $ \sum_ {k = 0} ^ { m \cdot n} \varphi_k \cdot (a \cdot b) ^ k = 0 $ with rational coefficients $ \varphi_k \in \mathbb Q $.

Here we observe that $ a \cdot b $ is a zero of the polynomial $ P = \sum_ {k = 0} ^ {m \cdot n} \varphi_k \cdot X ^ k \in \mathbb Z [X] $, i.e.~that $ a \cdot b $ is algebraic of degree $ \leq m \cdot n $.}%
\de{Der Beweis läuft ähnlich wie der für $a+b$, nur dass der Schritt mit dem binomischen Lehrsatz entfällt:

Sei $n$ der Grad von $a$ und $m$ der Grad von $b$.
Dann gibt es rationale Koeffizienten $\alpha_k$ und $\beta_k$ mit $a^n+\sum_{k=0}^{n-1}\alpha_k\cdot a^k=0$ und $b^m+\sum_{k=0}^{m-1}\beta_k\cdot b^k=0$.

Betrachte nun die Folge der Potenzen $(a\cdot b)^k$ für $k\in\mathbb N$. Immer wenn eine Potenz $a^r$ mit $r\geq n$ auftritt, kann man $a^n$ durch $-\sum_{k=0}^{n-1}\alpha_k\cdot a^k$ ersetzen, genauso bei $b^m$.

Die Potenzen $(a\cdot b)^k$ mit $0\leq k\leq m\cdot n$ lassen sich also alle als rationale Linearkombinationen der $a^r\cdot b^s$ mit $0\leq r<n$ und $0\leq s<m$ darstellen.

Folglich liegen diese $m\cdot n +1$ Potenzen alle in einem höchstens $m\cdot n$-dimen\-sionalen Vektorraum, sie müssen also linear abhängig sein und gibt es eine Darstellung $\sum_{k=0}^{m\cdot n}\varphi_k\cdot (a\cdot b)^k=0$ mit rationalen Koeffizienten $\varphi_k\in\mathbb Q$.

Hier erkennen wir, dass $a\cdot b$ eine Nullstelle des Polynoms $P=\sum_{k=0}^{m\cdot n}\varphi_k\cdot X^k\in\mathbb Z[X]$ ist, d.h. dass $a\cdot b$ algebraisch vom Grad $\leq m\cdot n$ ist.}
\end{proof}

\begin{lem}\label{\en{EN}lemwurz}
\en{If $a$ is algebraic, then $-a$, $\sqrt{a}$ and $1/a$ are also algebraic (the latter only if $a\neq 0$).}%
\de{Wenn $a$ algebraisch ist, dann sind auch $-a$, $\sqrt{a}$ und $1/a$ algebraisch (letzteres nur falls $a\neq 0$).}
\end{lem}
\begin{proof}
\en{Let $ n $ be the degree of $ a $.
Then there is an integer polynomial $ f $ with $ f (x) = x ^ n + \sum_ {k = 0} ^ {n-1} \alpha_k \cdot x ^ k $ and $ f (a) = 0 $.
The function $ g $ with $ g (x): = f (-x) $ is then also an integer polynomial and has the zero $ g (-a) = 0 $, so $ -a $ is also algebraic.

Secondly, the function $ k $ with $ k (x): = f (x ^ 2) $ is then also an integer polynomial and has the zero $ k (\sqrt {a}) = 0 $, so $ \sqrt { a} $ is algebraic as well.

Finallly, the function $ h $ with $ h (x): = f (1 / x) \cdot x ^ n = 1 + \sum_ {k = 0} ^ {n-1} \alpha_k \cdot x ^ { nk} $ is also an integer polynomial with zero $ h (1 / a) = 0 $ (if $ a \neq 0 $), so $ 1 / a $ is also algebraic.}%
\de{Sei $n$ der Grad von $a$.
Dann gibt es ein ganzzahliges Polynom $f$ mit $f(x)=x^n+\sum_{k=0}^{n-1}\alpha_k\cdot x^k$ und $f(a)=0$.
Die Funktion $g$ mit $g(x):=f(-x)$ ist dann ebenfalls ein ganzzahliges Polynom und hat die Nullstelle $g(-a)=0$, also ist auch $-a$ algebraisch.

Die Funktion $k$ mit $k(x):=f(x^2)$ ist dann ebenfalls ein ganzzahliges Polynom und hat die Nullstelle $k(\sqrt{a})=0$, also ist auch $\sqrt{a}$ algebraisch.

Weiter ist auch die Funktion $h$ mit $h(x):=f(1/x)\cdot x^n = 1 + \sum_{k=0}^{n-1}\alpha_k\cdot x^{n-k}$ ein ganzzahliges Polynom mit Nullstelle $h(1/a)=0$ (falls $a\neq 0$), also ist auch $1/a$ algebraisch.}
\end{proof}

\pagebreak

\section{\en{Constructible Points are Algebraic}\de{Konstruierbare Punkte sind algebraisch}}\label{\en{EN}kapkonstr}
\renewcommand{\leftmark}{\en{Constructible Points are Algebraic}\de{Konstruierbare Punkte sind algebraisch}}
\en{In this chapter we prove that all points which can be constructed in finitely many steps with compass and straightedge have algebraic coordinates (Thm.~\ref{\en{EN}theoalgkonstr}).}%
\de{In diesem Kapitel beweisen wir, dass alle Punkte, die in endlich vielen Schritten mit Zirkel und Lineal konstruierbar sind, algebraische Koordinaten haben (Thm.~\ref{\en{EN}theoalgkonstr}).}

\begin{defi}
\en{A point $P(x|y)$ is called ``algebraic'', if its coordinates $x$ and $y$ are algebraic numbers, i.e. zeros of polynomials from $\mathbb Z[X]$.}%
\de{Ein Punkt $P(x|y)$ heißt \glqq algebraisch\grqq, wenn seine Koordinaten $x$ und $y$ algebraische Zahlen sind, also Nullstellen von Polynomen aus $\mathbb Z[X]$.}
\end{defi}

\begin{theo}\label{\en{EN}theoalgkonstr}
\en{Let $M$ be a set of algebraic points, e.g. $M=\{A(0|0);B(1|0)\}$.
Then, starting with $M$ and using finitely many steps with compass and straightedge, one can construct only \textbf{algebraic} points.}%
\de{Sei $M$ eine Menge algebraischer Punkte, z.B. $M=\{A(0|0);B(1|0)\}$.
Dann kann man in endlich vielen Konstruktionsschritten, ausgehend von $M$, nur \textbf{algebraische} Punkte mit Zirkel und Lineal konstruieren.}
\end{theo}
\begin{proof}[\en{Proof sketch}\de{Beweisskizze}] \en{With a compass and ruler you can draw circles and straight lines and form their intersections. Circles are described by quadratic equations, straight lines by linear equations.

By induction on the number of construction steps we can prove the following: All linear or quadratic equations that occur have algebraic coefficients because the circles and straight lines are defined by algebraic points that have already been constructed.
So the coordinates of the intersections are solutions of linear or quadratic equations with algebraic coefficients and thus are algebraic.}%
\de{Mit Zirkel und Lineal kann man Kreise und Geraden zeichnen und ihre Schnittpunkte bilden. Kreise werden durch quadratische Gleichungen beschrieben, Geraden durch lineare Gleichungen.

Per vollständiger Induktion über die Anzahl der Konstruktionsschritte gilt: Alle auftretenden linearen bzw. quadratischen Gleichungen haben algebraische Koeffizienten, weil die Kreise und Geraden durch bereits konstruierte algebraische Punkte festgelegt werden.
Also sind die Koordinaten der Schnittpunkte Lösungen von linearen oder quadratischen Gleichungen mit algebraischen Koeffizienten und somit selbst wieder algebraisch.}
\end{proof}
\en{If this proof sketch is sufficient, you can jump directly to the squaring of the circle in Ch.~\ref{\en{EN}kapquadr} -- the following proof will only elaborate the details of this proof sketch.}%
\de{Wem diese Beweisskizze genügt, der kann direkt zur Quadratur des Kreises in Kap.~\ref{\en{EN}kapquadr} springen -- im hier folgenden Beweis werden nur noch die Details der Beweisskizze ausgearbeitet.}
\begin{proof}
\en{Each such construction with compass and ruler begins with the given points from $M$ and then consists of a finite sequence of the following operations, which may only build on previously constructed points, lines and circles:}%
\de{Jede solche Konstruktion mit Zirkel und Lineal beginnt mit den vorgegebenen Punkten aus $M$ und besteht danach aus einer endlichen Abfolge der folgenden Operationen, die jeweils nur auf bereits konstruierte Punkte, Geraden und Kreise aufbauen dürfen:}
\begin{enumerate}[leftmargin=8mm]
\en{\item Draw a line through two different points.
    \item Draw a circle around a point and through another point.
    \item Form the intersection of two straight lines.
    \item Form the intersection of a straight line with a circle.
    \item Form the intersection of two circles.}%
\de{\item Zeichne eine Gerade durch zwei verschiedene Punkte.
    \item Zeichne einen Kreis um einen Punkt und durch einen weiteren Punkt.
    \item Bilde den Schnittpunkt zweier Geraden.
    \item Bilde den Schnittpunkt einer Geraden mit einem Kreis.
    \item Bilde den Schnittpunkt zweier Kreise.}
\end{enumerate}
\en{Since steps (1) to (5) are only allowed to occur finitely often in every construction, we can carry out a induction on the number $n$ of construction steps. We simultaneously prove the following three statements:}%
\de{Da die Schritte (1) bis (5) in jeder Konstruktion nur endlich oft vorkommen dürfen, können wir eine vollständige Induktion über die Anzahl $n$ der Konstruktionsschritte durchführen. Wir beweisen dabei gleichzeitig die folgenden drei Aussagen:}
\begin{itemize}[leftmargin=8mm]
\en{\item[(a)] All constructed points have algebraic coordinates.
    \item[(b)] All constructed lines can be represented by equations of the form $a\cdot x + b\cdot y = c$ with algebraic coefficients $a$, $b$ and $c$, where $a^2 + b^2 \neq 0 $.
    \item[(c)] All constructed circles can be represented using equations of the form $(x-x_0)^2 + (y-y_0)^2 = r^2 $ with algebraic coefficients $x_0$, $y_0$ and $r\neq 0$.}%
\de{\item[(a)] Alle konstruierten Punkte haben algebraische Koordinaten.
    \item[(b)] Alle konstruierten Geraden können durch Gleichungen der Form $a\cdot x + b\cdot y = c$ mit algebraischen Koeffizienten $a$, $b$ und $c$ dargestellt werden, wobei $a^2+b^2\neq 0$.
    \item[(c)] Alle konstruierten Kreise können durch Gleichungen der Form $(x-x_0)^2+(y-y_0)^2=r^2$ mit algebraischen Koeffizienten $x_0$, $y_0$ und $r\neq 0$ dargestellt werden.}
\end{itemize}

\en{\emph{Start of induction} (after $n = 0$ construction steps): According to our premises, we start with the algebraic points from $M$, e.g. with $A(0|0)$ and $B(1|0)$.

\emph{Induction hypothesis:} The statements (a) to (c) apply to all points, circles and straight lines that can be constructed in $n$ construction steps.

\emph{Induction step:} We have to prove that the statements (a) to (c) also apply to all points, circles and straight lines that can be constructed in $n + 1$ steps. For this purpose we perform a case distinction: If \ldots}%
\de{\emph{Induktionsanfang} (nach $n=0$ Konstruktionsschritten): Nach Voraussetzung starten wir mit den algebraischen Punkten aus $M$, z.B. mit $A(0|0)$ und $B(1|0)$.

\emph{Induktionsvoraussetzung:} Die Aussagen (a) bis (c) gelten für alle Punkte, Kreise und Geraden, die man in $n$ Konstruktionsschritten konstruieren kann.

\emph{Induktionsschritt:} Zu zeigen ist, dass die Aussagen (a) bis (c) dann auch für alle Punkte, Kreise und Geraden gelten, die man in $n+1$ Schritten konstruieren kann. Hierzu eine Fallunterscheidung: Wenn im letzten Schritt \ldots}
\begin{enumerate}[leftmargin=8mm]
\en{\item \ldots a line was drawn through two different algebraic points, this line has algebraic coefficients (see Lemma \ref{\en{EN}lemgerade}).
     \item \ldots a circle has been drawn around an algebraic point and another algebraic point, this circle has algebraic coefficients (Lemma \ref{\en{EN}lemkreis}).
     \item \ldots the intersection of two lines with algebraic coefficients was formed, this intersection has algebraic coordinates (Lemma \ref{\en{EN}lemgeradegerade}).
     \item \ldots an intersection of a straight line and a circle (both with algebraic coefficients) was formed, this intersection has algebraic coordinates (Lemma \ref{\en{EN}lemschnittgeradekreis}).
     \item \ldots an intersection of two circles with algebraic coefficients was formed, this intersection has algebraic coordinates (Lemma \ref{\en{EN}lemkreiskreis}).}%
\de{ \item \ldots eine Gerade durch zwei verschiedene algebraische Punkte gezeichnet wurde, hat diese wieder algebraische Koeffizienten (siehe Lemma \ref{\en{EN}lemgerade}).
    \item \ldots ein Kreis um einen algebraischen Punkt und durch einen weiteren algebraischen Punkt gezeichnet wurde, hat dieser wieder algebraische Koeffizienten (Lemma \ref{\en{EN}lemkreis}).
    \item \ldots der Schnittpunkt zweier Geraden mit algebraischen Koeffizienten gebildet wurde, hat dieser wieder algebraische Koordinaten (Lemma \ref{\en{EN}lemgeradegerade}).
    \item \ldots ein Schnittpunkt einer Geraden und eines Kreises (beide mit algebraischen Koeffizienten) gebildet wurde, hat dieser wieder algebraische Koordinaten (Lemma \ref{\en{EN}lemschnittgeradekreis}).
    \item \ldots ein Schnittpunkt zweier Kreise mit algebraischen Koeffizienten gebildet wurde, hat dieser wieder algebraische Koordinaten (Lemma \ref{\en{EN}lemkreiskreis}).}
\end{enumerate}
\en{No other cases can occur because there are only these five possible operations (see above).
Apart from the proof of the lemmas \ref{\en{EN}lemgerade} to \ref{\en{EN}lemkreiskreis}, we have proved the statements (a) to (c) by induction on the number of construction steps and in particular we have proven the statement of Thm.~\ref{\en{EN}theoalgkonstr}.}%
\de{Weitere Fälle können nicht auftreten, weil es nur diese fünf Konstruktionsmöglichkeiten gibt (siehe oben).
Abgesehen vom Beweis der Lemmas \ref{\en{EN}lemgerade} bis \ref{\en{EN}lemkreiskreis} haben wir also die Aussagen (a) bis (c) per vollständiger Induktion über die Anzahl der Konstruktionsschritte bewiesen und insbesondere auch die Aussage des Thm.~\ref{\en{EN}theoalgkonstr}.}
\end{proof}

\en{The proofs of the lemmas \ref{\en{EN}lemgerade} to \ref{\en{EN}lemkreiskreis} are essentially based on the fact that the set of algebraic numbers is closed under addition, subtraction, multiplication, division and square root extraction (this was proved in Ch.~\ref{\en{EN}anhalg}, Prop.~\ref{\en{EN}satzalgabgeschl}).}%
\de{Die Beweise der Lemmas \ref{\en{EN}lemgerade} bis \ref{\en{EN}lemkreiskreis} stützen sich wesentlich darauf, dass die Menge der algebraischen Zahlen abgeschlossen unter Addition, Subtraktion, Multiplikation, Division und Quadratwurzelziehen ist (das wurde in Kap.~\ref{\en{EN}anhalg}, Satz \ref{\en{EN}satzalgabgeschl} bewiesen).}

\begin{lem}\label{\en{EN}lemgerade}
\en{If two different algebraic points $ P (x_1 | y_1) $ and $ Q (x_2 | y_2) $ are given,
then the line through $P$ and $Q$ can be represented by an equation $a\cdot x + b\cdot y = c$ with algebraic coefficients $a$, $b$ and $c$, where $a^2 + b^2 \neq 0$.}%
\de{Wenn zwei verschiedene algebraische Punkte $P(x_1|y_1)$ und $Q(x_2|y_2)$ gegeben sind,
dann kann die Gerade durch $P$ und $Q$ durch eine Gleichung $a\cdot x + b\cdot y = c$ mit algebraischen Koeffizienten $a$, $b$ und $c$ dargestellt werden, wobei $a^2+b^2\neq 0$ ist.}
\end{lem}
\begin{proof}
\en{A straight line through the points $ P (x_1 | y_1) $ and $ Q (x_2 | y_2) $ has $ \binom {y_1-y_2} {x_2-x_1} $ as normal vector and thus the equation $ a \cdot x + b \cdot y = c $ with coefficients}%
\de{Eine Gerade durch die Punkte $P(x_1|y_1)$ und $Q(x_2|y_2)$ hat $\binom{y_1-y_2}{x_2-x_1}$ als Normalenvektor und somit die Gleichung $a\cdot x + b\cdot y = c$ mit den Koeffizienten}
$$a=y_1-y_2,\qquad b=x_2-x_1\qquad \text{\en{and}\de{und}}\qquad c=x_2 y_1 - x_1 y_2.$$
\en{Because $ x_{1;2} $ and $ y_{1;2} $ are algebraic, Prop.~\ref{\en{EN}satzalgabgeschl} tells that the coefficients $ a $, $ b $ and $ c $ are also algebraic.
Next, the distance of $P$ and $Q$ can be expressed with $a$ and $b$:}%
\de{Weil nach Voraussetzung $x_{1;2}$ und $y_{1;2}$ algebraisch sind, folgt aus Satz \ref{\en{EN}satzalgabgeschl}, dass die Koeffizienten $a$, $b$ und $c$ ebenfalls algebraisch sind.
Außerdem gilt für den Abstand:}
$$d(P;Q)=\sqrt{(x_2-x_1)^2+(y_2-y_1)^2}=\sqrt{a^2+b^2}.$$
\en{Since we have $P\neq Q$, it holds $d(P;Q)\neq 0$ and thus $a^2+b^2\neq 0$.}%
\de{Nach Voraussetzung gilt $P\neq Q$, also $d(P;Q)\neq 0$ und somit $a^2+b^2\neq 0$.}
\end{proof}

\begin{lem}\label{\en{EN}lemkreis}
\en{If two different algebraic points $ M (x_0 | y_0) $ and $ P (x_1 | y_1) $ are given,
then the circle around $ M $ going through $ P $ can be expressed by the equation $ (x-x_0) ^ 2 + (y-y_0) ^ 2 = r ^ 2 $ with algebraic coefficients $ x_0 $, $ y_0 $ and $ r \neq 0 $.}%
\de{Wenn zwei verschiedene algebraische Punkte $M(x_0|y_0)$ und $P(x_1|y_1)$ gegeben sind,
dann kann der Kreis um $M$, der durch $P$ geht, durch die Gleichung $(x-x_0)^2+(y-y_0)^2=r^2$ mit algebraischen Koeffizienten $x_0$, $y_0$ und $r\neq 0$ dargestellt werden.}
\end{lem}
\begin{proof}
\en{A circle with center $ M (x_0 | y_0) $ going through $ P (x_1 | y_1) $ has the equation}%
\de{Ein Kreis mit Mittelpunkt $M(x_0|y_0)$, der durch $P(x_1|y_1)$ geht, hat die Gleichung}
$$(x-x_0)^2 + (y-y_0)^2 = r^2\qquad\text{\en{with}\de{mit}}\quad r=\sqrt{(x_1-x_0)^2+(y_1-y_0)^2}$$
\en{Because $ x_ {0; 1} $ and $ y_ {0; 1} $ are algebraic, Prop.~\ref{\en{EN}satzalgabgeschl} tells that the radius $ r $ is also algebraic. Finally, $ P \neq M $ yields $ r \neq 0 $.}%
\de{Weil nach Voraussetzung $x_{0;1}$ und $y_{0;1}$ algebraisch sind, folgt aus Satz \ref{\en{EN}satzalgabgeschl}, dass der Radius $r$ ebenfalls algebraisch ist. Aus $P\neq M$ folgt schließlich noch, dass $r\neq 0$ ist.}
\end{proof}

\begin{lem}\label{\en{EN}lemgeradegerade}
\en{If two lines with algebraic coefficients intersect, then the intersection also has algebraic coordinates.}
\de{Wenn sich zwei Geraden mit algebraischen Koeffizienten echt schneiden, dann hat der Schnittpunkt ebenfalls algebraische Koordinaten.}
\end{lem}
\begin{proof}
\en{Given the equations $a_1 x + b_1 y = c_1$ and $a_2 x + b_2 y = c_2$ of the two lines.}%
\de{Gegeben sind zwei Geradengleichungen $a_1 x + b_1 y = c_1$ und $a_2 x + b_2 y = c_2$.}

\en{If it holds $ a_ {1} b_ {2} -a_ {2} b_ {1} = 0 $, the two lines are parallel or coincidental and there is no proper intersection.

Otherwise, the coordinates of the intersection follow from Cramer's rule:
$$x_{s}={\frac {c_{1}b_{2}-c_{2}b_{1}}{a_{1}b_{2}-a_{2}b_{1}}}\qquad\text{and}\qquad y_{s}={\frac {a_{1}c_{2}-a_{2}c_{1}}{a_{1}b_{2}-a_{2}b_{1}}}$$
Prop.~\ref{\en{EN}satzalgabgeschl} tells that these coordinates are algebraic again.}%
\de{Falls $a_{1}b_{2}-a_{2}b_{1}=0$ ist, sind die beiden Geraden parallel oder sogar identisch und es gibt keinen echten Schnittpunkt.

Andernfalls folgen die Koordinaten des Schnittpunkts aus der Cramerschen Regel:
$$x_{s}={\frac {c_{1}b_{2}-c_{2}b_{1}}{a_{1}b_{2}-a_{2}b_{1}}}\qquad\text{und}\qquad y_{s}={\frac {a_{1}c_{2}-a_{2}c_{1}}{a_{1}b_{2}-a_{2}b_{1}}}$$
Aus Satz \ref{\en{EN}satzalgabgeschl} folgt, dass diese Koordinaten wieder algebraisch sind.}
\end{proof}

\begin{lem}\label{\en{EN}lemschnittgeradekreis}
\en{If the line $ ax + by = c $ and the circle $ (x-x_0) ^ 2 + (y-y_0) ^ 2 = r ^ 2 $ have algebraic coefficients $ a, b, c, x_0, y_0 $ and $ r $ and intersect, then the intersection points also have algebraic coordinates.}%
\de{Wenn die Gerade $ax+by=c$ und der Kreis $(x-x_0)^2+(y-y_0)^2=r^2$ algebraische Koeffizienten $a, b, c, x_0, y_0$ und $r$ haben und sich schneiden, dann haben die Schnittpunkte ebenfalls algebraische Koordinaten.}
\end{lem}
\begin{proof}
\en{We first shift the coordinates to $ \bar x = x - x_0 $ and $ \bar y = y - y_0 $ so that the equations become $ \bar x ^ 2 + \bar y ^ 2 = r ^ 2 $ and $ a \bar x + b \bar y = d $ with $ d = c - ax_0 - by_0 $. Then we insert the equation of the straight line into the circular equation and obtain:}%
\de{Wir verschieben zunächst die Koordinaten auf $\bar x = x - x_0$ und $\bar y = y - y_0$, sodass die Gleichungen $\bar x^2+\bar y^2=r^2$ und $a\bar x+b\bar y=d$ mit $d = c - ax_0 - by_0$ lauten. Dann setzen wir die Geradengleichung in die Kreisgleichung ein und erhalten:}
\begin{align*}
  b^2\bar x^2+b^2\bar y^2 = b^2r^2\quad\Longrightarrow\quad
   b^2\bar x^2 + (d-a\bar x)^2 = b^2r^2
\end{align*}
\en{Expanding this yields $\alpha \cdot \bar x^2 + \beta\cdot\bar x + \gamma = 0$ with}%
\de{Ausmultiplizieren und Zusammenfassen liefert $\alpha \cdot \bar x^2 + \beta\cdot\bar x + \gamma = 0$ mit}
$$\alpha=a^2+b^2,\qquad\beta=-2ad\qquad\text{\en{and}\de{und}}\qquad\gamma=d^2-b^2r^2$$
\en{and then the solutions $\bar x = \frac{-\beta\pm\sqrt{\beta^2-4\alpha\gamma}}{2\alpha}$.
Now Prop.~\ref{\en{EN}satzalgabgeschl} tells that $ \alpha, \beta, \gamma $ and thus also $ \bar x $ are algebraic (note $ \alpha = a ^ 2 + b ^ 2 \neq $ 0). In the exact same way you can also prove that $ \bar y $ is algebraic. Thus the shifted coordinates $ (x_s, y_s) $ of the intersections, namely $ (x_s, y_s) = (x_0 + \bar x, y_0 + \bar y) $, are also algebraic.}%
\de{und dann die Lösungen $\bar x = \frac{-\beta\pm\sqrt{\beta^2-4\alpha\gamma}}{2\alpha}$.
Aus Satz \ref{\en{EN}satzalgabgeschl} folgt, dass $\alpha, \beta, \gamma$ und somit auch $\bar x$ algebraisch sind (beachte $\alpha=a^2+b^2\neq 0$). Genauso kann man beweisen, dass $\bar y$ algebraisch ist und somit auch die zurückverschobenen Koordinaten $(x_s,y_s)$ der Schnittpunkte, nämlich $(x_s,y_s)=(x_0+\bar x, y_0+\bar y)$.}
\end{proof}

\begin{lem}\label{\en{EN}lemkreiskreis}
\en{If two circles $ (x-a_1) ^ 2 + (y-b_1) ^ 2 = r_1 ^ 2 $ and $ (x-a_2) ^ 2 + (y-b_2) ^ 2 = r_2 ^ 2 $ with algebraic coefficients $ a_1, a_2, b_1, b_2, r_1 \neq 0 $ and $ r_2 \neq 0 $ intersect, then the intersections also have algebraic coordinates.}%
\de{Wenn zwei Kreise $(x-a_1)^2+(y-b_1)^2=r_1^2$ und $(x-a_2)^2+(y-b_2)^2=r_2^2$ mit algebraischen Koeffizientenn $a_1, a_2, b_1, b_2, r_1\neq 0$ und $r_2\neq 0$ sich schneiden, dann haben die Schnittpunkte ebenfalls algebraische Koordinaten.}
\end{lem}
\begin{proof}
\en{First we calculate the difference between the two circular equations to eliminate the quadratic components:}%
\de{Zuerst bilden wir die Differenz der beiden Kreisgleichungen, um die quadratischen Anteile zu eliminieren:}
\begin{align*}
(x-a_1)^2-(x-a_2)^2+(y-b_1)^2-(y-b_2)^2&=r_1^2-r_2^2
\end{align*}
 \en{This gives a linear equation $ ax + by = c $ with coefficients $ a = -2a_1 + 2a_2 $, $ b = -2b_1 + 2b_2 $ and $ c = r_1 ^ 2-r_2 ^ 2-a_1 ^ 2 + a_2 ^ 2 -b_1 ^ 2 + b_2 ^ 2 $. These coefficients are algebraic again because of Prop.~\ref{\en{EN}satzalgabgeschl}. The intersections of the two circles thus correspond to the intersections of the one circle with this straight line. It follows from Lemma \ref{\en{EN}lemschnittgeradekreis} that the intersections have algebraic coordinates.}%
 \de{Das liefert eine lineare Gleichung $ax+by=c$ mit Koeffizienten $a=-2a_1+2a_2$, $b=-2b_1+2b_2$ und $c=r_1^2-r_2^2-a_1^2+a_2^2-b_1^2+b_2^2$. Diese Koeffizienten sind wegen Satz \ref{\en{EN}satzalgabgeschl} wieder algebraisch. Die Schnittpunkte der beiden Kreise entsprechen also den Schnittpunkten des einen Kreises mit dieser Gerade. Aus Lemma \ref{\en{EN}lemschnittgeradekreis} folgt dann, dass die Schnittpunkte algebraische Koordinaten haben.}
\end{proof}

\vspace*{5mm}
\section{\en{The Squaring of the Circle}\de{Die Quadratur des Kreises}}\label{\en{EN}kapquadr}
\renewcommand{\leftmark}{\en{The Squaring of the Circle}\de{Die Quadratur des Kreises}}
\en{In this chapter we prove that the Squaring of the Circle is impossible. For this purpose we gather our results from the chapters \ref{\en{EN}kaptransz}, \ref{\en{EN}anhalg} and \ref{\en{EN}kapkonstr}.}%
\de{In diesem Kapitel beweisen wir, dass die Quadratur des Kreises unmöglich ist. Dafür tragen wir unsere Ergebnisse aus den Kapiteln \ref{\en{EN}kaptransz}, \ref{\en{EN}anhalg} und \ref{\en{EN}kapkonstr} zusammen.}
\medskip

\begin{theo}\label{\en{EN}quadratur}
\en{Squaring the Circle is impossible. More precisely:}%
\de{Die Quadratur des Kreises ist unmöglich. Genauer gesagt: Wenn man}
\begin{itemize}[leftmargin=*]
    \en{\item Starting only from points with algebraic coordinates such as $ A (0 | 0) $ and $ B (1 | 0) $,
     \item and allowing only finitely many construction steps,
     \item and using only compass and straightedge,}%
    \de{\item nur von Punkten mit algebraischen Koordinaten wie z.B. $A(0|0)$ und $B(1|0)$ ausgeht
    \item und nur endlich viele Konstruktionsschritte zulässt
    \item und nur Zirkel und Lineal verwendet,}
\end{itemize}
\en{then you cannot construct a line segment of length $ \sqrt{\pi} $ and thus no square that has the same area as the unit circle.}%
\de{dann kann man kein Quadrat konstruieren, das den gleichen Flächeninhalt wie der Einheitskreis hat -- also keine Strecke der Länge $\sqrt{\pi}$.}
\end{theo}
\begin{proof}
\en{The area of the unit circle is $ \pi $. The desired square should therefore have the side length $ \sqrt{\pi} $.
If one had found such a construction, $ \sqrt{\pi} $ would have to be algebraic (Thm.~\ref{\en{EN}theoalgkonstr}). Then Prop.~\ref{\en{EN}satzalgabgeschl} would prove that $\pi=\sqrt{\pi}\cdot\sqrt{\pi}$ is also algebraic.
But in Thm.~\ref{\en{EN}pitransz} we proved that $ \pi $ is transcendental, so squaring the circle is impossible.}%
\de{Der Einheitskreis hat den Flächeninhalt $\pi$. Das gewünschte Quadrat müsste also die Seitenlänge $\sqrt{\pi}$ haben.
Wenn man eine solche Konstruktion gefunden hätte, müsste $\sqrt{\pi}$ algebraisch sein (Thm.~\ref{\en{EN}theoalgkonstr}) und wegen Satz \ref{\en{EN}satzalgabgeschl} müsste dann auch $\pi=\sqrt{\pi}\cdot\sqrt{\pi}$ algebraisch sein.
In Thm.~\ref{\en{EN}pitransz} haben wir aber bewiesen, dass $\pi$ transzendent ist. Also ist die Quadratur des Kreises unmöglich.}
\end{proof}

\vspace*{5mm}
\bibliographystyle{plain}
{\raggedright
\renewcommand{\refname}{\en{References}\de{Literatur}}
\label{\en{EN}literat}

\label{\en{EN}ende}
}
\fancyhead[OL]{\emph{\en{References}\de{Literatur}}}



  \vfill\pagebreak
  
  
  \enfalse
  \invisiblepart{Deutsche Version}

\end{document}